\documentclass[11pt,a4paper]{amsart}
\usepackage{todonotes}
\usepackage[T1]{fontenc}
\usepackage[utf8]{inputenc}
\usepackage{lmodern}
\usepackage{amsmath,amssymb,amsthm}
\usepackage[margin=27mm,headheight=14pt]{geometry}
\usepackage{microtype}
\usepackage{xcolor}
\definecolor{linkgreen}{RGB}{28,82,65}
\usepackage[colorlinks=true,linkcolor=linkgreen,urlcolor=linkgreen,bookmarksopen=true]{hyperref}
\newtheorem{theorem}{Theorem}[section]
\newtheorem{proposition}[theorem]{Proposition}
\newtheorem{lemma}[theorem]{Lemma}
\newtheorem{corollary}[theorem]{Corollary}
\newtheorem{swtheorem}{Theorem}

\theoremstyle{remark}
\newtheorem{remark}[theorem]{Remark}

\newcommand{\ord}{\operatorname{ord}}

\numberwithin{equation}{section}
\allowdisplaybreaks[1]

\title[Almost-holomorphic harmonic map flow]{Refined analysis on finite-time singularities of almost-holomorphic harmonic map flow}

\author[Y. He]{Yihan He}
\address{School of Mathematical Sciences, Xiamen University\\
	Fujian, 361005, P. R. China}
\email{heyihan@stu.xmu.edu.cn}

\author[Z. Li]{Ze Li}
\address{School of Mathematics and Statistics, Ningbo University,
	Ningbo 315211, Zhejiang, China}
\email{lize@nbu.edu.cn}

\author[C. Song]{Chong Song}
\address{School of Mathematical Sciences, Xiamen University\\
	Fujian, 361005, P. R. China}
\email{songchong@xmu.edu.cn}



\thanks{Y. He and C. Song are partially supported by NSFC (No. 12371061 and No. 12526102) and Natural Science Foundation of Fujian Province (No. 2026J011003). Z. Li is partially supported by the NSFC (No. 12422110 and No. 12371244) and by Ningbo Youth Leading Talent Project (No. 2024QL046).}

\date{\today}

\begin{document}
	\maketitle

	\begin{abstract}
		Building on the work of Song--Waldron~\cite{SongWaldron2024}, we establish stronger regularity and refined blow-up analysis results for almost-holomorphic harmonic map flow from a closed surface into a compact K\"ahler manifold with nonnegative holomorphic bisectional curvature. At each finite-time singularity, we improve the outer energy scale to $O((T-t)/\sqrt{|\ln(T-t)|})$. We also prove that both the stress-energy tensor and the tension field converge strongly in $L^2$ to those of the body map, by ruling out concentration of their $L^2$ energy near the singular points. Consequently, the body map belongs to $W^{2,2}$. Finally, an explicit gap theorem for the antiholomorphic energy of nonholomorphic harmonic spheres yields an explicit admissible threshold for the initial antiholomorphic energy under which these conclusions hold.
	\end{abstract}

	\tableofcontents
	
	\section{Introduction}

	\subsection{Background}
	
	The analysis of singularities in harmonic map heat flow from surfaces has developed over more than four decades and has played an important role in the study of geometric evolution equations. Struwe~\cite{Struwe1985} established global weak existence and regularity away from finitely many singular points, while Chang--Ding--Ye~\cite{ChangDingYe1992} constructed the first examples of finite-time singularities. During the 1990s, the work of Ding--Tian~\cite{DingTian1995}, Qing~\cite{Qing1995}, Wang~\cite{Wang1996}, Qing--Tian~\cite{QingTian1997} and Lin--Wang~\cite{LinWang1998} led to a detailed understanding of energy concentration and bubbling in harmonic map heat flow from surfaces. Through energy identities and bubble-tree analysis, these works showed that, at finite or infinite singular times, the flow admits a decomposition along suitably chosen time sequences into a body map and finitely many harmonic spheres called bubbles. The limiting energy is exactly the sum of the energies of these components, with no energy lost in the neck regions.  For a comprehensive account of the classical theory of harmonic maps and their heat flows, we refer the reader to the monograph of Lin--Wang~\cite{LinWang2008}.
	
	These advances established a detailed picture of energy concentration, but left important aspects of the flow near singularities unresolved. One concerns convergence along the full time variable and uniqueness of the bubbling configuration. Another problem is the continuity of the body map and the no-neck property at finite-time singularities.
	
	For general smooth targets, Topping's winding examples~\cite{Topping2004a} show that the body map can be discontinuous and the bubbles	can depend on the chosen time sequence. These examples motivated his conjecture on the continuity and no-neck for real-analytic targets, which are still largely open.\footnote{During the preparation of the current manuscript, we noticed that an arXiv preprint of Rupflin--Waldron--Woodward~\cite{RupflinWaldronWoodward2026} claimed a solution to Topping's conjecture in the case that at most one bubble forms at each singular point.} In another paper~\cite{Topping2024b}, Topping proved both properties if the Dirichlet energy $E(u(t))$ has a H\"older continuous extension to $[0, T]$.  
	
	For spherical targets, these questions have received more complete answers under equivariant symmetry. In the scalar equivariant class of maps from $\mathbb R^2$ to $S^2$, Jendrej--Lawrie~\cite{JendrejLawrie2023} established a unique bubble decomposition along the full time variable, at both finite and infinite time. At finite singular times, this also gives continuity of the body map and the no-neck property. Without the symmetry assumption, Jendrej--Lawrie--Schlag~\cite{JendrejLawrieSchlag2025} proved a weaker result that the flow approaches the family of multi-bubble configurations along the full time variable. Consequently, every time sequence approaching the maximal existence time admits a bubbling subsequence, although their result does not establish uniqueness of the limiting bubble configuration or control oscillation on the necks.
	
	The K\"ahler setting provides another useful structure to address these questions, through the holomorphic and antiholomorphic decomposition of the differential. This structure allows assumptions on one component without requiring the total energy to be small. Topping~\cite{Topping1997} exploited this structure to establish uniform-in-time rigidity at infinite time for harmonic map heat flow from $S^2$ to $S^2$, when the limiting body map and all bubbles are holomorphic, or all antiholomorphic. Liu--Yang~\cite{LiuYang2010} extended this rigidity theory to flows from $S^2$ into compact K\"ahler manifolds with nonnegative holomorphic bisectional curvature.
	
	This structure also leads to improved regularity at finite-time singularities. Under a smallness assumption on the antiholomorphic or holomorphic energy, Song--Waldron~\cite{SongWaldron2024} proved continuity of the body map and the no-neck property for flows into such K\"ahler targets, together with a quantitative upper bound for the outer energy scale. These results provide the starting point for the present paper; we recall them after fixing our notation.

	\subsection{Settings and known results}
	
	Let $(\Sigma,g,j)$ be a closed oriented Riemannian surface and let $(N,h,J)$ be a compact K\"ahler manifold. Unless explicitly stated otherwise, we assume $N$ has nonnegative holomorphic bisectional curvature. We also fix an isometric embedding $N\hookrightarrow\mathbb R^m$ when needed.
	
	For $0<T<\infty$, we consider a smooth harmonic map heat flow $u:\Sigma\times[0,T)\longrightarrow N$ satisfying
	\[
	\partial_tu=\mathcal T(u):=\operatorname{tr}_g\nabla du.
	\]
	The energy density, energy, and stress-energy tensor are denoted by
	\[
	e(u)=\frac12|du|^2,\qquad
	E(u;U)=\int_U e(u)\,dV,\qquad
	S(u)=u^*h-e(u)g.
	\]
	We write $E(u)=E(u;\Sigma)$ and $E_0=E(u(0))$. The real antiholomorphic differential is
	\[
	\Phi(u)=\frac12(du+J\circ du\circ j),
	\qquad e_{\bar\partial}(u)=\frac12|\Phi(u)|^2,
	\qquad E_{\bar\partial}(u;U)=\int_U e_{\bar\partial}(u)\,dV.
	\]
	
	Let $u_T$ be the terminal body map and let $\mathcal S\subset\Sigma$ be the finite singular set at terminal time $T$. Thus $u(t)\rightharpoonup u_T$ in $W^{1,2}(\Sigma)$ and $u(t)\to u_T$ smoothly on compact subsets of $\Sigma\setminus\mathcal S$.

	For $p\in\Sigma$ and $0<a<R$ below the injectivity radius, let $B_r(p)$ denote a geodesic ball and $U_a^R(p)=B_R(p)\setminus\overline{B_a(p)}$ denote an annulus. Following Song--Waldron, define
	
	\begin{equation}\label{eq:outer-energy-scale}
		\lambda_{\varepsilon,\rho,p}(t)
		=\inf\left\{\lambda\in[0,\rho]:
		\sup_{\lambda<r<\rho}E(u(t);U_{r/2}^r(p))<\varepsilon\right\}.
	\end{equation}
	In particular, if no $\lambda<\rho$ satisfies the condition in \eqref{eq:outer-energy-scale}, then $\lambda=\rho$ does so vacuously.
	
	The following statement records the main conclusions of \cite[Theorem~1.1 and Corollary~4.6]{SongWaldron2024}.
	
	\begin{swtheorem}[Song-Waldron~\cite{SongWaldron2024}]\label{thm:song-waldron}
		There is a constant $\delta_{\mathrm{SW}}=\delta_{\mathrm{SW}}(N)>0$ such that the following holds. Let $u:\Sigma\times[0,T)\to N$ be a smooth harmonic map heat flow with $T<\infty$ and
		\begin{equation}\label{assmption:SW}
			E_{\bar\partial}(u(0))<\delta_{\mathrm{SW}}.
		\end{equation}
		Then, for every sequence $t_n\uparrow T$, there is a subsequence admitting a finite bubble-tree decomposition with body $u_T$ and holomorphic spheres $\omega_1,\ldots,\omega_m$, such that
		\begin{itemize}
			\item The body map satisfies $u_T\in C^\alpha(\Sigma,N)$ for every $0<\alpha<1$.
			\item The rescaled maps converge strongly in $W^{1,2}_{\mathrm{loc}}$ on each regular bubble region and
			\begin{equation}\label{eq:song-waldron-energy-identity}
				E_*:=\lim_{t\uparrow T}E(u(t))=E(u_T)+\sum_{\alpha=1}^m E(\omega_\alpha).
			\end{equation}
			\item There is no neck oscillation: the images of the bubbles and the body map form a connected set.
			\item For every blow-up point $p\in\mathcal S$, every $\varepsilon>0$, and all sufficiently small $\rho>0$,
			\begin{equation}\label{eq:song-waldron-outer-scale}
				\lambda_{\varepsilon,\rho,p}(t)=O(T-t)\qquad(t\uparrow T).
			\end{equation}
		\end{itemize} 
	\end{swtheorem}
	
	The same results hold with antiholomorphic energy replaced by holomorphic energy in assumption \eqref{assmption:SW}. A map satisfying \eqref{assmption:SW} is called an almost-holomorphic map. Since the antiholomorphic energy is nonincreasing, a harmonic map heat flow starting with an almost-holomorphic map remains almost-holomorphic throughout its smooth existence interval and is called an almost-holomorphic flow. 
	
	Theorem~\ref{thm:song-waldron} provides improved regularity at finite-time singularities under the almost-holomorphic assumption. These results were extended to the dissipative Landau--Lifshitz flow from $\mathbb R^2$ to $S^2$ by Li--Song~\cite{LiSong2025} and to two-dimensional Yang-Mills-Higgs heat flow by Chen-Song~\cite{ChenSong2026}. Further progress on the regularity of the body map has been made by Waldron. For general compact targets, he proved that a strict type-II bound on the outer energy scale implies H\"older continuity of the body map~\cite{Waldron2023b}. More recently, he established a weighted \L{}ojasiewicz criterion for continuity of the body map and the no-neck property, and verified it for flows from $\mathbb R^2$ into $S^2$ without an almost-holomorphic assumption~\cite{Waldron2025}. His related work on \L{}ojasiewicz inequalities gives polynomial convergence at infinite time for flows from $S^2$ to $S^2$ whose body map is nonconstant~\cite{Waldron2023a}.
	
	In this paper, we aim to answer further questions left by Theorem~\ref{thm:song-waldron}, including an improved estimate of the energy scales, the non-concentration of the $L^2$-energy of the stress and tension field, and an explicit admissible antiholomorphic-energy threshold in \eqref{assmption:SW}.
	
	\subsection{Main results}
	
	For a harmonic map heat flow $u:\Sigma\times[0,T)\to N$, we consider the uniform pointwise bound
	\[
	|\Phi(u)|\le B
	\quad\text{on }\Sigma\times[0,T)
	\quad\text{for some }B<\infty.
	\tag{H}\label{assumption:H}
	\]
	Together with the nonnegative holomorphic bisectional curvature, these are the standing assumptions of our main results. In fact, the small antiholomorphic energy assumption implies \eqref{assumption:H}, and Theorem~\ref{thm:song-waldron} remains valid under \eqref{assumption:H} (see Proposition~\ref{prop:song-waldron-under-bounded-differential} below). In particular, every bubble arising at a singularity is holomorphic. This common orientation is also consistent with the hypotheses in the rigidity results of Topping~\cite{Topping1997} and Liu--Yang~\cite{LiuYang2010}.

	Our first result improves the outer-scale bound $O(T-t)$ in Theorem~\ref{thm:song-waldron}. In the spherical blow-up constructions of Rapha\"el--Schweyer~\cite{RaphaelSchweyer2013} and D\'avila--del Pino--Wei~\cite{DavilaDelPinoWei2020}, the concentration scale is proportional to $(T-t)/|\ln(T-t)|^2$, suggesting a sharper bound than \eqref{eq:song-waldron-outer-scale}. The following theorem provides a positive answer.
	
	\begin{theorem}\label{thm:outer-scale-estimates}
		Under the standing assumptions, for every blow-up point $p\in\mathcal S$, every $\varepsilon>0$, and all sufficiently small $\rho>0$,
		\begin{equation}\label{eq:outer-scale-log-improvement}
			\lambda_{\varepsilon,\rho,p}(t)
			=O\!\left(\frac{T-t}{\sqrt{|\ln(T-t)|}}\right)\qquad(t\uparrow T).
		\end{equation}
		In particular, $\lambda_{\varepsilon,\rho,p}(t)=o(T-t)$.
	\end{theorem}
	
	Our second result establishes strong $L^2$ convergence of both the stress-energy tensor and the tension field at the singular time. An interesting question raised by W. Ding in his lecture notes~\cite{Ding1998} asks whether the $L^2$ norm of the tension field remains bounded as a finite-time singularity is approached. Under our assumptions, we answer this question affirmatively by ruling out concentration of the $L^2$ tension energies.
	
	\begin{theorem}\label{thm:terminal-convergence}
		Under the standing assumptions, 
		\begin{itemize}
			\item The stress-energy tensors satisfy
			\begin{equation}\label{eq:terminal-stress-convergence}
				S(u(t))\longrightarrow S(u_T)
				\quad\text{strongly in }L^2(\Sigma)\quad(t\uparrow T).
			\end{equation}
			\item The tension field satisfies
			\begin{equation}\label{eq:terminal-tension-convergence}
				\mathcal T(u(t))\longrightarrow \mathcal T(u_T)
				\quad\text{strongly in }L^2(\Sigma)\quad(t\uparrow T). 
			\end{equation}
			In particular, $L^2$-tension energy is uniformly bounded on $[0,T)$ and
			$$ \lim_{t\uparrow T}\|\mathcal T(u(t))\|_2 = \|\mathcal T(u_T)\|_2. $$
		\end{itemize}
	\end{theorem}
	
	As a consequence of the tension convergence in Theorem~\ref{thm:terminal-convergence}, we obtain the $W^{2,2}$-regularity of the body map and the one-sided first-order expansions of the energy functional. In particular, the time function $t\mapsto E(u(t))$, extended by $E(T)=E_*$, is Lipschitz continuous on $[0,T]$, verifying the H\"older continuity criterion of Topping~\cite{Topping2024b} under our assumptions.
	
	\begin{corollary}\label{cor:terminal-expansions}
		Under the standing assumptions, 
		\begin{itemize}
			\item The body map satisfies
			\[
			u_T\in W^{2,2}(\Sigma,N).
			\]
			\item The energy functional satisfies
			\begin{equation*}
				E(u(t))=E_*+(T-t)\|\mathcal T(u_T)\|_2^2+o(T-t).
			\end{equation*}
		\end{itemize}		
	\end{corollary}

	Our last result provides an explicit gap for the energy of a holomorphic sphere and for the antiholomorphic energy of a nonholomorphic harmonic sphere, which gives a lower bound for the optimal constant $\delta_{\mathrm{SW}}$ in Theorem~\ref{thm:song-waldron}. The gap theorem can be compared with the classical result of Wood~\cite{Wood1979} for harmonic maps into $\mathbb{CP}^n$ and the gap theorem for Yang-Mills fields of Gursky--Kelleher--Streets~\cite{GurskyKellehrtStreets2018}. To state the result, write
	\[
	H_*:=\max_{q\in N,\,0\ne X\in T_qN}H_N(X),
	\qquad \delta_N:=\frac{4\pi}{H_*},
	\]
	where $H_N$ is the holomorphic sectional curvature, normalized to be one on the unit round sphere. As there are no harmonic spheres if $H_N$ is nonpositive, we assume $H_*>0$ in the following theorem. Note this result is independent of our standing assumptions stated above. 
	
	\begin{theorem}\label{thm:antiholomorphic-energy-threshold}
		Let $N$ be a compact K\"ahler manifold with nonnegative holomorphic bisectional curvature and $H_*>0$. Then every nonconstant harmonic sphere $\omega:S^2\to N$ satisfies exactly one of the following:
		\begin{itemize}
			\item $\omega$ is holomorphic and 
			\begin{equation}\label{eq:thm-holo-gap}
				E_{\partial}(\omega)\ge\delta_N.
			\end{equation}
			\item $\omega$ is nonholomorphic and
			\begin{equation}\label{eq:thm-gap}
				E_{\bar\partial}(\omega)\ge\delta_N.
			\end{equation}
		\end{itemize}
	\end{theorem}
	
	\begin{remark}
		The gap $\delta_N$ is optimal in the sense that it equals $4\pi$ if $N$ is the standard sphere or more generally $\mathbb{CP}^n$. Note that harmonic spheres in $\mathbb{CP}^n$ for $n\ge2$ need not be holomorphic or antiholomorphic, as exhibited by Eells--Wood~\cite{EellsWood1983}. 
	\end{remark}
	
	In fact, we prove a more general antiholomorphic energy gap theorem for closed surfaces in Section~\ref{sec:antiholomorphic-energy-gap}. In Appendix~\ref{sec:gap-rigidity}, we record a geometric characterization of the equalities, together with a Riemannian analog of the gap theorem.
	
	It follows from \eqref{eq:thm-gap} that the optimal threshold in Theorem~\ref{thm:song-waldron} satisfies 
	\[ \delta_{\mathrm{SW}}\ge \delta_N. \]
	Combining Theorem~\ref{thm:antiholomorphic-energy-threshold} with the preceding results gives the following consequence of the explicit energy threshold.
	
	\begin{corollary}\label{cor:gap-flow-app}
		Let $N$ satisfy the geometric assumptions of Theorem~\ref{thm:antiholomorphic-energy-threshold}. If a smooth harmonic map heat flow $u:\Sigma\times[0,T)\to N$ satisfies
		\begin{equation}\label{eq:initial-gap-threshold}
			E_{\bar\partial}(u(0))<\delta_N,
		\end{equation}
		then \eqref{assumption:H} holds. Consequently, the conclusions of Theorems~\ref{thm:outer-scale-estimates} and~\ref{thm:terminal-convergence} and Corollary~\ref*{cor:terminal-expansions} hold at every finite terminal time of a smooth interval of the flow.
	\end{corollary}

	Finally, we remark that all the above results hold if we replace the antiholomorphic differential $\Phi$ in assumption~\eqref{assumption:H} with its holomorphic counterpart. Indeed, it follows by simply reversing the orientation of the domain. With this understood, we stick with our assumption \eqref{assumption:H} in the following context.
	
	\
	
	The rest of the paper is organized as follows. Section~\ref{sec:preliminaries} derives the evolution equation for $\Phi$, its Bochner formula and stress identities, and records the analytic inputs and bubble-tree construction from Song--Waldron. Section~\ref{sec:outer-scale-estimates} combines a local tension integral estimate with logarithmic local dissipation to prove Theorem~\ref{thm:outer-scale-estimates}. Section~\ref{sec:bubble-estimates} proves rigidity of normalized ancient fields and deduces tension and stress estimates on bubble regions. Section~\ref{sec:neck-estimates} establishes the tension estimate on neck regions. Section~\ref{sec:terminal-convergence} combines these estimates with the energy identity to prove Theorem~\ref{thm:terminal-convergence} and Corollary~\ref{cor:terminal-expansions}. The energy gap and Theorem~\ref{thm:antiholomorphic-energy-threshold} are proved in Section~\ref{sec:antiholomorphic-energy-gap}. Appendix~\ref{sec:annular-estimates} contains the scalar estimates used on the necks, and Appendix~\ref{sec:gap-rigidity} gives the rigidity part and a Riemannian analog of the gap theorem.

	\subsection*{AI Declaration.}
	The authors used ChatGPT (OpenAI) as a research assistance tool to explore proof strategies, check calculations, and assist with language editing. The authors take full responsibility for the mathematical content of this paper.

	\section{Preliminaries and analytic inputs}
	\label{sec:preliminaries}
	
	We first derive basic identities and equations along the harmonic map heat flow. Then we recall analytic inputs from Song--Waldron including $\varepsilon$-regularity and outer energy scale propagation estimates, and record the bubble-tree construction. For later analysis, we will use the real form of the antiholomorphic differential rather than complex notation throughout the paper.
	
	\subsection{Real splitting, tension, and stress}
	
	Let $(\Sigma,g,j)$ be a closed oriented Riemannian surface and let $(N,h,J)$ be a compact K\"ahler manifold. Given a sufficiently regular map $u:\Sigma \to N$, the energy density, energy, and stress-energy tensor are denoted by
	\[
	e(u)=\frac12|du|^2,\qquad
	E(u)=\int_\Sigma e(u)\,dV,\qquad
	S(u)=u^*h-e(u)g.
	\]
	Denoting the connection induced by that of $\Sigma$ and $N$ by $\nabla$, the tension field of $u$ is
	\[
	\mathcal{T}(u) =\operatorname{tr}_g\nabla du.
	\]
	
	Using the complex structures, we decompose
	\[ du = \Phi(u) + \Psi(u),\]
	where
	\begin{equation}\label{eq:real-splitting}
		\Phi(u):=\frac12(du+J\circ du\circ j),\qquad
		\Psi(u):=\frac12(du-J\circ du\circ j).
	\end{equation}
	The forms $\Phi, \Psi$ are the real forms of the antiholomorphic and holomorphic differentials of $u$, satisfying
	\[ \Phi\circ j=-J\circ \Phi, \quad \Psi\circ j=J\circ \Psi, \] 
	and
	\[ \langle \Phi,\Psi\rangle=0. \]
	
	We retain the customary energy subscripts:
	\[
	e_{\bar\partial}=\tfrac12|\Phi|^2,\qquad
	e_\partial=\tfrac12|\Psi|^2.
	\]
	Their integrals are denoted by $E_{\bar\partial}$ and $E_\partial$. As usual, we have 
	\[
	e=e_\partial+e_{\bar\partial},\qquad
	u^*\omega_N=(e_\partial-e_{\bar\partial})\,dV,
	\]
	where $\omega_N$ is the K\"ahler form of $N$. On a closed surface, $\kappa=\int_\Sigma u^*\omega_N$ is invariant under smooth homotopy and
	\begin{equation*}
		E=2E_{\bar\partial}+\kappa=2E_\partial-\kappa.
	\end{equation*}
	
	\begin{lemma}\label{lem:real-first-order-identities}
		The real differential $\Phi$ satisfies
		\begin{equation}\label{eq:real-first-order}
			2\operatorname{div} \Phi=\mathcal T,\qquad
			d^\nabla \Phi=-\tfrac12J\mathcal T\,dV.
		\end{equation}
		Consequently $\Phi$ is a harmonic bundle-valued 1-form when $u$ is harmonic. 
	\end{lemma}
	\begin{proof}
		Choose a positively oriented orthonormal frame $e_2=je_1$ and put $X=\Phi(e_1),Y=\Psi(e_1)$. Then the symmetry $(\nabla du)_{12}=(\nabla du)_{21}$ and $\nabla j=\nabla J=0$ give
		\[
		2\operatorname{div}\Phi
		=\mathcal T+J\bigl((\nabla du)_{12}-(\nabla du)_{21}\bigr)
		=\mathcal T.
		\]
		Also $(d^\nabla \Phi)_{12}=-J\nabla_1X-\nabla_2X
		=-J\operatorname{div}\Phi$, proving \eqref{eq:real-first-order}.
	\end{proof}
	
	\begin{lemma}
		\label{lem:real-stress-identities}
		The stress tensor has the bilinear expression
		\begin{equation}\label{eq:stress-expression}
			S = h(\Phi\otimes\Psi) + h(\Psi\otimes\Phi).
		\end{equation}
		Hence it satisfies
		\begin{equation}\label{eq:stress-pointwise-bound}
			|S|^2 \le8e_\partial e_{\bar\partial}
		\end{equation}
		and
		\begin{equation}\label{eq:stress-divergence}
			\operatorname{div}S=\langle\mathcal T,du\rangle,
		\end{equation}
	\end{lemma}
	\begin{proof}
		Since $du=\Phi + \Psi$ and $\Phi^*h=e_{\bar{\partial}}g, \Psi^*h=e_{\partial}g$, we have for vectors $v, w$,
		\begin{align*}
			u^*h(v,w) &= h(du(v), du(w)) \\
			&= h(\Phi(v), \Phi(w)) + h(\Phi(v), \Psi(w)) + h(\Psi(v), \Phi(w)) + h(\Psi(v), \Psi(w))\\
			&= h(\Phi(v), \Psi(w)) + h(\Psi(v), \Phi(w)) + e(u)\langle v, w \rangle.
		\end{align*} 
		This proves \eqref{eq:stress-expression}. Then \eqref{eq:stress-pointwise-bound} and the identity \eqref{eq:stress-divergence} follow by direct calculation; see also \cite[Lemma~2.2]{SongWaldron2024}.
	\end{proof}
	
	For target vectors $X,Y \neq 0$, we write the holomorphic bisectional curvature as
	\[
	B_N(X, Y)=\frac{h(R^N(X,JX)JY,Y)}{|X|^2|Y|^2},
	\]
	and the holomorphic sectional curvature as
	\[
	H_N(X)=\frac{h(R^N(X,JX)JX,X)}{|X|^4},\qquad X\ne0,
	\]
	which equals one on the unit round sphere. The products $B_N(X,Y)e_\partial e_{\bar\partial}$ and $H_N(X)e_{\bar\partial}^2$ are understood by continuous extension when either vector vanishes.

	For a bundle-valued one-form $\xi$ of either type, $\xi\circ j=\pm J\circ \xi$, define the target endomorphism
	\begin{equation}\label{eq:real-curvature-operator}
		\mathcal R_{\xi} =J\circ R^N(X,JX),\qquad X=\xi(e_1).
	\end{equation}
	It acts on bundle-valued forms through their target factor. The endomorphism \eqref{eq:real-curvature-operator} is independent of the positively oriented orthonormal frame. It is self-adjoint, commutes with $J$, and satisfies 
	$$ |\mathcal R_{\xi} |\le C_N|\xi|^2.$$
	If $N$ has nonnegative holomorphic bisectional curvature, then $\mathcal R_{\xi} \ge0$. 
	
	Moreover, setting $X=\Phi(e_1),Y=\Psi(e_1)$, then
	\begin{equation*}
		\tfrac12\langle\mathcal R_{\Psi}  \Phi,\Phi\rangle
		=B_N(X,Y)e_\partial e_{\bar\partial},
	\end{equation*}
	and
	\begin{equation*}
		\tfrac12\langle\mathcal R_{\Phi}  \Phi,\Phi\rangle
		=H_N(X)e_{\bar\partial}^2.
	\end{equation*}

	\subsection{Evolution equation for \texorpdfstring{$\Phi$}{Phi} and Bochner formula}
	
	Now suppose $u:\Sigma\times[0,T)\to N$ is a smooth harmonic map heat flow. The following point-wise energy identity was observed by Song-Waldron~\cite{SongWaldron2024}.
	
	\begin{lemma}
		Along harmonic map heat flow,
		\begin{equation}\label{eq:local-energy-stress}
			\partial_t e+|\mathcal T|^2
			=\operatorname{div}\langle\mathcal T,du\rangle
			=\operatorname{div}\operatorname{div}S.
		\end{equation}
		For a smooth time-independent cutoff $\chi$,
		\begin{equation}\label{eq:cutoff-stress-identity}
			\frac{d}{dt}\int_\Sigma\chi e\,dV
			+\int_\Sigma\chi|\mathcal T|^2\,dV
			=\int_\Sigma\langle S,\nabla^2\chi\rangle\,dV.
		\end{equation}
		In particular,
		\begin{equation}\label{eq:energy-dissipation}
			\frac{dE}{dt}=-\|\mathcal T\|_2^2,\qquad
			\frac{dE_{\bar\partial}}{dt}
			=\frac{dE_\partial}{dt}=-\tfrac12\|\mathcal T\|_2^2.
		\end{equation}
	\end{lemma}
	
	A cutoff argument using \eqref{eq:local-energy-stress} gives the following local energy inequality, which transfers local energy bounds to nearby times.
	
	\begin{lemma}\label{lem:local-energy-comparison}
		Let $0\le\chi\le1$ be a smooth time-independent cutoff. For $0\le s\le t<T$, write
		\[ \mathcal D:=\int_s^t\|\mathcal T(u(\tau))\|_2^2\,d\tau. \]
		Then
		\begin{equation*}
			\left|\int_\Sigma\chi\bigl(e(u(t))-e(u(s))\bigr)\,dV\right|
			\le \mathcal{D} +C\|\nabla\chi\|_\infty\sqrt{(t-s)E_0\mathcal D}.
		\end{equation*}
		
		In particular, if $t-s\le R^2$ and $|\nabla\chi|\le C/R$, the right-hand side is bounded by $\mathcal D+C\sqrt{E_0\mathcal D}$.
		
	\end{lemma}
	
	The following evolution equation of $\Phi$ will play a key role in our analysis.
	
	\begin{proposition}\label{prop:antiholomorphic-evolution}
		The real antiholomorphic differential $\Phi=\Phi(u)$ satisfies
		\begin{equation}\label{eq:antiholomorphic-evolution}
			(\nabla_t-\Delta_\nabla)\Phi
			=-K_{\Sigma}\Phi-\mathcal R_{\Psi}  \Phi+\mathcal R_{\Phi}  \Phi.
		\end{equation}
	\end{proposition}
	\begin{proof}
		Let $\delta^\nabla=(d^\nabla)^*$ and $\Pi_-\xi=\tfrac12(\xi+J\circ\xi\circ j)$. Since $j,J$ are parallel, $\nabla_t\Phi=\Pi_-\nabla\mathcal T$. On a surface, \eqref{eq:real-first-order} gives
		\[
		\delta^\nabla \Phi=-\tfrac12\mathcal T,\qquad
		d^\nabla \Phi=-\tfrac12J\mathcal T\,dV,
		\]
		and therefore
		\[
		(d^\nabla\delta^\nabla+\delta^\nabla d^\nabla)\Phi
		=-\Pi_-\nabla\mathcal T=-\nabla_t \Phi. 
		\]
		
		Choose a positively oriented orthonormal frame with $e_2=je_1$, and write $X=\Phi(e_1)$ and $Y=\Psi(e_1)$. By definition,
		\[
		\Phi_1=X,\qquad \Phi_2=-JX,\qquad
		u_1=X+Y,\qquad u_2=JY-JX.
		\]
		Commuting covariant derivatives yields
		\[
		\bigl((d^\nabla\delta^\nabla+\delta^\nabla d^\nabla)\Phi\bigr)_j
		=-\Delta_\nabla \Phi_j+K_{\Sigma}\Phi_j+
		\sum_iR^N(u_i,u_j)\Phi_i.
		\]
		Thus we get
		\begin{equation}\label{eq:real-weitzenbock}
			(\nabla_t-\Delta_\nabla)\Phi_j=-K_{\Sigma}\Phi_j-
			\sum_iR^N(u_i,u_j)\Phi_i.
		\end{equation}
		
		For the last curvature term in \eqref{eq:real-weitzenbock}, using \eqref{eq:real-splitting} and the identity $R^N(X,JY)=R^N(Y,JX)$, we have
		\[
		R^N(u_1,u_2)=R^N(Y,JY)-R^N(X,JX).
		\]
		By \eqref{eq:real-curvature-operator} and the commutation of curvature operator with $J$, the first component of the curvature term satisfies
		\[
		\begin{aligned}
			-\sum_iR^N(u_i,u_1)\Phi_i
			&=-R^N(u_2,u_1)(-JX)\\
			&=-JR^N(u_1,u_2)X\\
			&=-\mathcal R_\Psi X+\mathcal R_\Phi X.
		\end{aligned}
		\]
		Applying $-J$ gives the corresponding identity for the second component. This proves \eqref{eq:antiholomorphic-evolution}.
		
	\end{proof}
	
	The Bochner formula for $\Phi$ follows directly.
	
	\begin{proposition}
		\begin{equation*}
			\tfrac12(\partial_t-\Delta)|\Phi|^2
			=-|\nabla \Phi|^2-K_{\Sigma}|\Phi|^2
			-\langle\mathcal R_{\Psi} \Phi,\Phi\rangle
			+\langle\mathcal R_{\Phi} \Phi,\Phi\rangle.
		\end{equation*}
		Equivalently,
		\begin{equation}\label{eq:antiholomorphic-bochner}
			\tfrac12(\partial_t-\Delta)e_{\bar\partial}
			=-\tfrac12|\nabla \Phi|^2-K_{\Sigma}e_{\bar\partial}
			-B_N(X, Y)e_{\partial}e_{\bar\partial}+H_N(X)e_{\bar\partial}^2.
		\end{equation}
	\end{proposition}

	\begin{remark}
		The complex form of the Bochner formula was obtained by Liu--Yang~\cite{LiuYang2010} and Song--Waldron~\cite{SongWaldron2024}. Only for comparison, write $g=2\sigma_{\mathrm{SW}}^2|dz|^2$ and use the complex notation of \cite[Section~2]{SongWaldron2024}. Then
		\begin{equation*}
			\begin{gathered}
				\Phi=2\operatorname{Re}(\bar\partial u),\qquad
				\Psi=2\operatorname{Re}(\partial u),\\
				e_{\bar\partial}=|\Phi|^2/2=|\bar\partial u|^2
				=\sigma_{\mathrm{SW}}^{-2}|w_{\bar z}|^2,\qquad
				|\nabla \Phi|^2=2|\nabla\bar\partial u|^2,\\
				q_1=-B_N(X,Y)e_{\partial}e_{\bar\partial},\qquad
				q_2=\tfrac12\langle\mathcal R_{\Phi} \Phi,\Phi\rangle
				=H_N(X)e_{\bar\partial}^2.
			\end{gathered}
		\end{equation*}
		Thus \eqref{eq:antiholomorphic-bochner} is their split Bochner identity with the real normalization made explicit. Similarly \eqref{eq:stress-expression} is their identity $S=2\operatorname{Re}Q_{\mathrm{SW}}$ in Lemma~2.2, where $Q_{\mathrm{SW}}=2\langle w_z,w_{\bar z}\rangle dz^2$. The energy densities therefore agree exactly, so their analytic estimates below apply without a change of energy convention.
	\end{remark}
	
	The following lemma records the scaling properties of the tension and stress energy.
	
	\begin{lemma}\label{lem:tension-stress-scaling}
		Let $p\in\Sigma$ and $F(y)=z(p)+\lambda y$ in an oriented coordinate chart $z$, and set $v(y,s)=u(F(y),t_*+\lambda^2s)$, $g_\lambda=\lambda^{-2}F^*g$. For $t=t_*+\lambda^2s$, we have
		\begin{align*}
			\int_{F(K)}|\mathcal T_g(u(t))|^2\,dV_g
			&=\lambda^{-2}\int_K|\mathcal T_{g_\lambda}(v(s))|^2\,dV_{g_\lambda},
			\\
			\int_{F(K)}|S_g(u(t))|_g^2\,dV_g
			&=\lambda^{-2}\int_K|S_{g_\lambda}(v(s))|_{g_\lambda}^2\,dV_{g_\lambda};
		\end{align*}
		and
		\begin{align*}
			\int_{t_*-\lambda^2L}^{t_*}\int_{F(K)}|\mathcal T_g(u(t))|^2\,dV_gdt
			&=\int_{-L}^0\int_K|\mathcal T_{g_\lambda}(v(s))|^2\,dV_{g_\lambda}ds,
			\\
			\int_{t_*-\lambda^2L}^{t_*}\int_{F(K)}|S_g(u(t))|_g^2\,dV_gdt
			&=\int_{-L}^0\int_K|S_{g_\lambda}(v(s))|_{g_\lambda}^2\,dV_{g_\lambda}ds.
		\end{align*}
	\end{lemma}

	\subsection{Regularity and scale propagation}
	
	Under the nonnegative holomorphic bisectional curvature assumption, we recall the regularity and scale estimates in Song-Waldron~\cite{SongWaldron2024} below.

	The first proposition is the standard $\varepsilon$-regularity for harmonic map heat flow.
	
	\begin{proposition}[{\cite[Theorem~3.4 and Corollary~3.6]{SongWaldron2024}}]\label{prop:interior-parabolic-regularity}
		There is $\varepsilon_0>0$ such that the following holds. Suppose a smooth harmonic map heat flow is defined on $B_R(p)\times[s-R^2,s]$, with $R$ below a fixed geometric radius, and
		
		\[
		\sup_{s-R^2\le t\le s}E(u(t);B_R(p))\le\varepsilon\le\varepsilon_0.
		\]
		
		For every nonnegative integer $k$,
		
		\begin{equation*}
			\sup_{B_{R/2}(p)\times[s-R^2/4,s]}
			|\nabla^kdu|\le C_k R^{-k-1}\sqrt\varepsilon.
		\end{equation*}
		
		Interior mixed spacetime derivatives satisfy the corresponding parabolic bounds.
		
	\end{proposition}
	
	The next proposition is the $\varepsilon$-regularity for the antiholomorphic differential.
	
	\begin{proposition}[{\cite[Theorem~3.5]{SongWaldron2024}}]\label{prop:antiholomorphic-regularity}
		If $E_{\bar\partial}(u(0))<\delta_{\mathrm{SW}}$, with the threshold in Theorem~\ref{thm:song-waldron} chosen sufficiently small, then $\Phi(u)$ is uniformly bounded on $\Sigma\times[t_0,T)$ for every $t_0\in(0,T)$. More locally, let $R$ be below a fixed geometric radius and let $t_1<T-R^2$. If
		
		\[
		E_{\bar\partial}(u(t_1);B_R(p))
		+\sup_{t_1\le t<T}E_{\bar\partial}(u(t);U_{R/2}^R(p))
		\le\delta<\delta_0,
		\]
		
		for a sufficiently small $\delta_0>0$, then
		
		\[
		\sup_{B_{R/2}(p)\times[t_1+R^2/2,T)}|\Phi(u)|
		\le\frac{C\sqrt\delta}{R}.
		\]	
	\end{proposition}
	
	Under \eqref{assumption:H}, the neck estimate of \cite[Theorem~6.4]{SongWaldron2024} takes the following form.
	
	\begin{proposition}\label{prop:two-tail-neck-estimate}
		Assume \eqref{assumption:H}. For every $0<\beta<1$ there is $\varepsilon_\beta>0$ such that, for $0<\varepsilon<\varepsilon_\beta$, there exist $r_\beta,d_\beta>0$ with the following property. If $0<R\le r_\beta$, $t_0\le t_1<T-4R^2$, and
		\[
		\int_{t_1}^T\|\mathcal T(u(s))\|_2^2\,ds\le d_\beta,
		\]
		then, for $t_1+4R^2\le t<T$ and $2\lambda_{\varepsilon,R,p}(t)\le r=d_g(x,p)\le R/2$,
		\begin{equation}\label{eq:two-tail-neck-estimate}
			r|du(x,t)|\le C_\beta\sqrt\varepsilon
			\left[\left(\frac{\lambda_{\varepsilon,R,p}(t)}r\right)^\beta
			+\left(\frac rR\right)^\beta\right].
		\end{equation}
		Here $C_\beta$ depends only on $\beta$ and the geometries; $r_\beta,d_\beta$ may also depend on $B,E_0,\varepsilon$.
	\end{proposition}
	\begin{proof}
		Choose $\nu\in(1/2,1)$ with $\nu>\beta$ and $\nu(2\nu-1)>\beta$, and let $\varepsilon_\beta$ be the threshold in \cite[Theorem~6.4]{SongWaldron2024}. Enlarge $B$ to be positive. Apply that theorem with $\tau=t_1$, $\rho=R$, $\rho_1=2R$, and $\delta=C_\Sigma B^2(2R)^2$. By \eqref{assumption:H},
		\[
		\sup_{t_1\le s<T}E_{\bar\partial}(u(s);B_{2R}(p))\le\delta.
		\]
		Its radius and small-energy conditions hold for $r_\beta$ small, since they reduce to $2R<R_0$, $\delta<\delta_0$, and $CB\sqrt{E_0}R<\varepsilon$. Taking $d_\beta\le c\varepsilon^2/E_0$ gives its tension hypothesis; the zero-energy case is immediate. The stated time range includes the waiting time $(2R)^2$. The theorem gives \eqref{eq:two-tail-neck-estimate} with exponents $\nu$ and $\nu(2\nu-1)$; lowering both to $\beta$ proves the claim.
	\end{proof}
	
	We will also need the following proposition on propagation of outer energy scale.
	
	\begin{proposition}[{\cite[Theorem~4.4]{SongWaldron2024}}]\label{prop:outer-scale-propagation}
		Let $0<a\le R$ be sufficiently small geometric radii and suppose $\sup_{t_1\le t\le t_2}\|S(u(t))\|_{L^2(B_R(p))}\le\mathfrak s$. If
		
		\[
		\sup_{t_1\le t\le t_2}E(u(t);B_R(p))
		\le E(u(t_2);B_a(p))+\frac\varepsilon2,
		\]
		
		then, for $t_1\le t\le t_2$,
		
		\begin{equation*}
			\lambda_{\varepsilon,R,p}(t)
			\le\max\left\{4a,\frac{C\mathfrak s(t_2-t)}\varepsilon\right\}.
		\end{equation*}
		
	\end{proposition}
	
	Under our assumption \eqref{assumption:H}, Theorem~\ref{thm:song-waldron} of Song--Waldron can be restated as follows.
	
	\begin{proposition}
		\label{prop:song-waldron-under-bounded-differential}
		Let $u:\Sigma\times[0,T)\to N$ be a smooth harmonic map heat flow, with $T<\infty$, into a compact K\"ahler manifold with nonnegative holomorphic bisectional curvature. If \eqref{assumption:H} holds, then all conclusions of Theorem~\ref{thm:song-waldron} hold.
	\end{proposition}
	\begin{proof}
		This is a direct consequence in view of the proof of Theorem~\ref{thm:song-waldron} in \cite{SongWaldron2024}. For completeness, we recall the key steps in the proof using \eqref{assumption:H}. 
		
		By the real stress identity of Lemma~\ref{lem:real-stress-identities}, assumption \eqref{assumption:H} gives a uniform bound of the $L^2$ stress energy. Hence \cite[Theorem~6.8 and Corollary~6.9]{SongWaldron2024} apply to any sequence $t_n\uparrow T$, yielding the bubble-tree subconvergence with energy identity \eqref{eq:song-waldron-energy-identity} and no-neck. Assumption \eqref{assumption:H} also implies nonconcentration of antiholomorphic energy; \cite[Corollary~6.5]{SongWaldron2024} then gives every H\"older exponent below one. Finally, the energy scale bound \eqref{eq:song-waldron-outer-scale} follows from \cite[Corollary~4.5]{SongWaldron2024} with $q=2$.
	\end{proof}

	We finish this subsection with the logarithmic cutoff estimates used below.
	
	\begin{lemma}\label{lem:logarithmic-cutoff}
		Let $D_a\subset D_R$ be Euclidean disks in $\mathbb R^2$, with $0<a<R/4$, and put $\ell=\ln(R/a)$. There is a function $\chi_{a,R}\in C_c^\infty(D_R)$ such that $0\le\chi_{a,R}\le1$, $\chi_{a,R}=1$ on $D_a$, and
		\begin{equation}\label{eq:logarithmic-cutoff-derivatives}
			\int_{\mathbb R^2}|\nabla\chi_{a,R}|^2\,dx
			+\int_{\mathbb R^2}|\Delta\chi_{a,R}|\,dx
			\le\frac C\ell.
		\end{equation}
		Moreover,
		\begin{equation}\label{eq:logarithmic-cutoff-mass}
			\int_{\mathbb R^2}\chi_{a,R}\,dx\le\frac{CR^2}\ell.
		\end{equation}
		The constant is independent of the center and radii. The same estimates hold in a fixed conformal coordinate chart with respect to $g$, with a constant depending only on an upper bound for its conformal factor.
	\end{lemma}
	\begin{proof}
		Fix a smooth nonincreasing function $\vartheta:\mathbb R\to[0,1]$ which equals one on $(-\infty,1/4]$ and zero on $[3/4,\infty)$. Writing $r$ for the distance to the common center, set
		\[
		\chi_{a,R}(x)=\vartheta\!\left(\frac{\ln(r/a)}\ell\right)
		\quad(r>0),\qquad \chi_{a,R}(0)=1,
		\]
		This is smooth and compactly supported in $D_R$. Radial differentiation and the substitution $s=\ln(r/a)/\ell$ give
		\[
		\int|\nabla\chi_{a,R}|^2\,dx
		=\frac{2\pi}\ell\int_0^1|\vartheta'(s)|^2\,ds,
		\qquad
		\int|\Delta\chi_{a,R}|\,dx
		=\frac{2\pi}\ell\int_0^1|\vartheta''(s)|\,ds.
		\]
		Since $\vartheta(s)\le C(1-s)$ for $s\le1$, we have $\chi_{a,R}(x)\le C\ell^{-1}\ln(R/r)$ for $0<r<R$. Thus
		\[
		\int\chi_{a,R}\,dx
		\le\frac{2\pi C}\ell\int_0^R r\ln(R/r)\,dr
		\le\frac{CR^2}\ell.
		\]
		
		For the conformal metric $g=e^{2\sigma}|dx|^2$ in dimension two, we have
		\[
		|\nabla_g\chi|_g^2=e^{-2\sigma}|\nabla\chi|^2,
		\qquad
		\Delta_g\chi=e^{-2\sigma}\Delta\chi,
		\qquad
		dV_g=e^{2\sigma}dx.
		\]
		Consequently, the estimate \eqref{eq:logarithmic-cutoff-derivatives} is unchanged. Moreover, since $\chi\ge0$ is supported in $D_R$,
		\[
		\int\chi\,dV_g
		\le \left(\sup_{D_R}e^{2\sigma}\right)\int\chi\,dx
		\le \frac{CR^2}{\ell},
		\]
		where the last constant also depends on an upper bound for the conformal factor.
	\end{proof}
	
	\subsection{Bubble tree construction and energy identity}
	\label{subsec:tree-construction}
	
	Next we describe the outside-in bubble-tree construction of \cite[Theorem~6.8]{SongWaldron2024}, which originated with Parker~\cite{Parker1996}, and identify the regular body region, the bubble regions, and the neck regions. Assume~\eqref{assumption:H}, which gives the uniform $L^2$ stress bound required in that construction. For interested readers, we refer to \cite{LiWang2010adv} for an inside-out construction.
	
	Fix an arbitrary sequence $t_n\uparrow T$. At each concentration point $p\in\mathcal S$, choose a small fixed coordinate disk containing no other concentration point. Start with an outer energy scale $\lambda_n^\circ$ at $p$, with a small threshold $\epsilon<\varepsilon_0$. Recenter at the center of mass of the energy in $B_{\lambda_n^\circ}(p)$ and define $\lambda_n$ to be the outer energy scale about this new center, with half the original outer radius. Their construction gives
	\[
	0<\lambda_n\le C\lambda_n^\circ\longrightarrow0.
	\]
	Proposition~\ref{prop:outer-scale-propagation} gives $\lambda_n^\circ=O(T-t_n)$. Rescale space by $\lambda_n$ and time by $\lambda_n^2$, keeping $t_n$ as the terminal time.
	
	Away from finitely many concentration points, the rescaled flows converge to a stationary harmonic map. The vanishing rescaled dissipation identifies stationarity, and removable singularities complete the limit to a harmonic sphere. It is holomorphic because the rescaled antiholomorphic differential is bounded by $CB\lambda_n$. The sphere may be constant. In that case there are at least two distinct concentration points: if there were only one, the center-of-mass normalization would put it at the origin; the energy on a fixed annulus about the origin would then tend to zero, contradicting the normalization of the outer energy scale to one. This is the reason for retaining a constant intermediate component, often called a \emph{ghost bubble}.
	
	Repeat the same construction at every remaining concentration point of the rescaled sequence. A nonconstant sphere carries at least the fixed energy gap $\varepsilon_0$, while a constant component divides the remaining concentration into at least two parts, each carrying at least $\varepsilon_0$. Thus the energy available at any subsequent concentration point decreases by at least $\varepsilon_0$. The construction terminates after finitely many steps. 
	
	Index the resulting spherical components, constant ones included, by $\alpha=1,\ldots,m$. Write $z_{\alpha,n}$, $\lambda_{\alpha,n}$, and $\omega_\alpha$ for their coordinate centers, energy scales, and limiting maps; thus $z_{\alpha,n}=z(p_{\alpha,n})$ for points $p_{\alpha,n}\in\Sigma$. For each $\alpha$, let $\pi(\alpha)$ denote the component on which the concentration producing $\omega_\alpha$ occurred. The index $0$ denotes the body, and $\lambda_{0,n}=1$ in each fixed coordinate disk. For $\pi(\alpha)\ne0$,
	\begin{equation*}
		\frac{\lambda_{\alpha,n}}{\lambda_{\pi(\alpha),n}}\longrightarrow0,
		\qquad
		\frac{z_{\alpha,n}-z_{\pi(\alpha),n}}
		{\lambda_{\pi(\alpha),n}}\longrightarrow q_\alpha.
	\end{equation*}
	When $\pi(\alpha)=0$, the center tends to its concentration point $q_\alpha=p$ on the body. Components arising directly from the same component have distinct attachment points $q_\alpha$. 
	
	Choose $L$ large enough to contain all finite attachment points, and $d>0$ small enough that their closed $2d$-disks are disjoint on each component. Define the regular body region by 
	\[ \mathcal R(d):= \Sigma\setminus\left(\cup_{p\in\mathcal S}B_{d/2}(p)\right).\]
	Then
	\[ u(x, t_n) \longrightarrow u_T(x), \text{~in~} C^\infty(\mathcal R(d)).\]
	Put
	\[
	K_\alpha(L,d)=\overline{D_{2L}}\setminus
	\bigcup_{\pi(\gamma)=\alpha}D_{d/2}(q_\gamma),
	\]
	so that
	\[ u(z_{\alpha, n}+\lambda_{\alpha, n}y,t_n) \longrightarrow \omega_\alpha(y), \text{~in~} C^\infty(K_\alpha(L,d)).\]
	Denote the bubble region about $z_{\alpha, n}$ by
	\[
	\mathcal B_{\alpha,n}(L,d)=z_{\alpha,n}+\lambda_{\alpha,n}K_\alpha(L,d).
	\]
	Define the corresponding neck region by
	\begin{equation*}
		\mathcal N_{\alpha,n}(L,d)=
		\{L\lambda_{\alpha,n}<|x-z_{\alpha,n}|
		<d\lambda_{\pi(\alpha),n}\}.
	\end{equation*}
	
	All these pieces are allowed to overlap and together they cover $\Sigma$. Since there are finitely many components, all choices of $L,d$ below can be made simultaneously. The energy identity \eqref{eq:song-waldron-energy-identity} asserts that the energy on the neck region vanishes.
	
	\begin{lemma}
		\label{lem:component-cover}
		For $L$ sufficiently large and $d$ sufficiently small, the regular body piece, the sets $\mathcal B_{\alpha,n}(L,d)$, and the annuli $\mathcal N_{\alpha,n}(L,d)$ cover $\Sigma$ for all large $n$. Moreover, for every connecting annulus,
		\begin{equation*}
			\lim_{d\downarrow0}\lim_{L\to\infty}\limsup_n
			E(u(t_n);\mathcal N_{\alpha,n}(L,d))=0.
		\end{equation*}
	\end{lemma}

	\section{Refined estimates for the outer energy scale}\label{sec:outer-scale-estimates}
	
	In this section, we prove the improved outer energy scale estimate in Theorem~\ref{thm:outer-scale-estimates} under the standing assumptions. By Proposition~\ref{prop:song-waldron-under-bounded-differential} and \eqref{eq:song-waldron-outer-scale}, we have, for every $p\in\mathcal S$, every $\varepsilon>0$, and all sufficiently small $\rho>0$,
	\begin{equation}\label{eq:preliminary-linear-scale}
		\lambda_{\varepsilon,\rho,p}(t)\le C(T-t)
		\qquad(t\uparrow T).
	\end{equation}
	We improve this estimate by establishing a logarithmic control of the spacetime integral of $|\nabla\Phi|^2$.
	
	\subsection{Upper bound of the outer energy scale}
	
	The following lemma is somewhat standard and provides a bound of the outer energy scale by the local spacetime integral of the tension field.
	
	\begin{proposition}\label{prop:local-tension-scale}
		Under the standing assumptions, for each $p\in\mathcal S$ and $\varepsilon>0$, there exist $\rho_*>0$ and $t_*<T$ such that, for every $0<\rho\le\rho_*$ and $t_*<t<T$,
		\begin{equation}\label{eq:local-tension-scale}
			\lambda^2\le C(T-t)
			\int_t^T\int_{B_{\lambda/2}(p)}|\mathcal T(u(s))|^2\,dV\,ds,
			\qquad \lambda=\lambda_{\varepsilon,\rho,p}(t).
		\end{equation}
		The constant $C$ only depends on $\varepsilon$ and $E_0$.
	\end{proposition}
	\begin{proof}
		We first identify the limit energy measure. The uniform energy bound gives subsequential weak limits of $e(u(t))\,dV$ as $t\uparrow T$. For every fixed smooth cutoff $\chi$, Lemma~\ref{lem:local-energy-comparison} and the finiteness of the total dissipation imply
		\[
		\left|\int_\Sigma\chi\bigl(e(u(t))-e(u(s))\bigr)\,dV\right|\longrightarrow0
		\qquad(s,t\uparrow T).
		\]
		Thus all subsequential limits agree. Smooth convergence away from $\mathcal S$ identifies the full-time limit as
		\begin{equation}\label{eq:terminal-measure}
			e(u(t))\,dV\rightharpoonup
			\mu_T=e(u_T)\,dV+\sum_{p\in\mathcal S}\Theta_p\delta_p.
		\end{equation}
		
		Fix $p\in\mathcal S$ and $\varepsilon>0$. Choose $\rho_*>0$ below the injectivity radius so that $\overline{B_{\rho_*}(p)}$ contains no other singular point and
		\[
		E(u_T;B_{\rho_*}(p))<\frac\varepsilon8.
		\]
		Hence, for all sufficiently late $t$,
		\begin{equation}\label{eq:local-energy-bound}
		E(u(t);B_\rho(p))\le E(u(t);B_{\rho_*}(p))
		\le\Theta_p+\frac\varepsilon4
		\qquad(0<\rho\le\rho_*).
		\end{equation}
		Fix such $\rho,t$ and set $\lambda=\lambda_{\varepsilon,\rho,p}(t)>0$.

		Continuity of the annular energy in its radius and the definition of $\lambda$ give
		\begin{equation}\label{eq:outer-scale-threshold-annulus}
			E(u(t);U_{\lambda/2}^{\lambda}(p))\ge\varepsilon.
		\end{equation}
		Combining \eqref{eq:local-energy-bound} with \eqref{eq:outer-scale-threshold-annulus} yields
		\begin{equation}\label{eq:inner-energy-bound}
			E(u(t);B_{\lambda/2}(p))\le\Theta_p-\frac{3\varepsilon}4.
		\end{equation}
		
		Choose a cut-off function $0\le\chi\le1$ supported in $B_{\lambda/2}(p)$, satisfying $\chi(p)=1$ and  $|\nabla\chi|\le C/\lambda$. Applying \eqref{eq:cutoff-stress-identity} with this fixed cutoff, integrating from $t$ to $T$, and using \eqref{eq:terminal-measure} and \eqref{eq:inner-energy-bound}, we obtain
		\begin{align*}
			\frac{3\varepsilon}{4}
			&\le
			\int_\Sigma\chi\,d\mu_T
			-\int_\Sigma\chi e(u(t))\,dV\\
			&\le
			\int_t^T\int_\Sigma
			\langle S(u(s)),\nabla^2\chi\rangle\,dV\,ds\\
			&=
			-\int_t^T\int_\Sigma
			\langle \mathcal{T}(u(s)),du(\nabla\chi)\rangle\,dV\,ds\\
			&\le\frac{C\sqrt{E_0}}\lambda
			\int_t^T\|\mathcal T(u(s))\|_{L^2(B_{\lambda/2}(p))}\,ds.
		\end{align*}
		Cauchy--Schwarz in time then proves \eqref{eq:local-tension-scale}.
	\end{proof}
	
	\subsection{Logarithmic control of local dissipation}
	
	The next estimate gives the desired logarithmic control of the local spacetime integral of $|\nabla \Phi|^2$.
	
	\begin{lemma}
		\label{lem:log-local-dissipation}
		Under the standing assumptions, for concentric disks $D_a\subset D_R$ in a fixed conformal chart, with $0<a<R/4$, and $t_0\le t<T$, we have
		\begin{equation}\label{eq:log-dissipation}
			\int_t^T\int_{D_a}
			|\nabla\Phi|^2\,dV\,ds
			\le
			\frac{C\bigl(R^2+(T-t)+(T-t)R^2\bigr)}{\ln(R/a)}.
		\end{equation}
		Here $C$ depends only on $B$ and the fixed geometries.
	\end{lemma}
	
	\begin{proof}
		Take $\chi=\chi_{a,R}$ from Lemma~\ref{lem:logarithmic-cutoff} in the given conformal chart. Thus $\chi=1$ on $D_a$, $\chi$ is supported in $D_R$, and
		\begin{equation}\label{eq:log-cutoff}
			\int_\Sigma|\Delta\chi|\,dV\le\frac C{\ln(R/a)},
			\qquad
			\int_\Sigma\chi\,dV\le\frac{CR^2}{\ln(R/a)}.
		\end{equation}
		
		Fix $t<t_1<T$. Multiplying \eqref{eq:antiholomorphic-bochner} by $2\chi$ and integrating by parts in space gives
		\[
		\begin{aligned}
			&\frac{d}{ds}\int_\Sigma\chi e_{\bar\partial}\,dV
			+\int_\Sigma\chi
			\left(
			|\nabla\Phi|^2
			+2B_N(X,Y)e_\partial e_{\bar\partial}
			\right)\,dV\\
			&\qquad=
			\int_\Sigma e_{\bar\partial}\Delta\chi\,dV
			-2\int_\Sigma\chi K_\Sigma e_{\bar\partial}\,dV
			+2\int_\Sigma\chi H_N(X)e_{\bar\partial}^2\,dV.
		\end{aligned}
		\]
		By \eqref{assumption:H} and the curvature bounds,
		\[
		e_{\bar\partial}\le\frac{B^2}{2},
		\qquad
		2|K_\Sigma|e_{\bar\partial}
		+2|H_N(X)|e_{\bar\partial}^2\le C.
		\]
		Integrating from $t$ to $t_1$, discarding the nonnegative terminal term, and then letting $t_1\uparrow T$, we obtain
		\begin{equation*}
			\begin{aligned}
				&\int_t^T\int_\Sigma\chi
				\left(
				|\nabla\Phi|^2
				+2B_N(X,Y)e_\partial e_{\bar\partial}
				\right)\,dV\,ds\\
				&\qquad\le
				C\int_\Sigma\chi\,dV
				+C(T-t)\int_\Sigma|\Delta\chi|\,dV
				+C(T-t)\int_\Sigma\chi\,dV.
			\end{aligned}
		\end{equation*}
		Since $\chi=1$ on $D_a$ and $B_N(X,Y)\ge0$, \eqref{eq:log-cutoff} gives
		\[
			\int_t^T\int_{D_a}
			\left(
			|\nabla\Phi|^2+B_N(X,Y)e_\partial e_{\bar\partial}
			\right)\,dV\,ds
			\le
			\frac{C\bigl(R^2+(T-t)+(T-t)R^2\bigr)}{\ln(R/a)},
		\]
		which proves \eqref{eq:log-dissipation}.
	\end{proof}
	
	The first-order identity \eqref{eq:real-first-order} gives the pointwise estimate
	\begin{equation}\label{eq:tension-gradient-bound}
		|\mathcal T|^2=4|\operatorname{div}\Phi|^2\le8|\nabla\Phi|^2.
	\end{equation}
	Thus Lemma~\ref{lem:log-local-dissipation} also controls the spacetime integral of $|\mathcal T|^2$ on $D_a$.
	
	\subsection{Proof of Theorem~\ref{thm:outer-scale-estimates}}
	
	\begin{proof}[Proof of Theorem~\ref{thm:outer-scale-estimates}]
		Fix $p\in\mathcal S$, $\varepsilon>0$, and a sufficiently small $\rho>0$. Write
		\[
		\delta=T-t,\qquad\lambda=\lambda_{\varepsilon,\rho,p}(t).
		\]
		By \eqref{eq:preliminary-linear-scale}, $\lambda\le C_0\delta$ for all sufficiently late $t$. In a fixed conformal chart centered at $p$, choose a fixed $K>0$ large enough that $B_{\lambda/2}(p)\subset D_a$, where
		\[
		a=K\delta,\qquad R=\sqrt\delta.
		\]
		For sufficiently small $\delta$, the disk $D_R$ lies in the chart, $a<R/4$, and
		\[
		\ln(R/a)=\frac12|\ln\delta|-\ln K\ge\frac14|\ln\delta|.
		\]
		These radii are fixed over the time interval $[t,T)$. Lemma~\ref{lem:log-local-dissipation} and \eqref{eq:tension-gradient-bound} imply
		\[
		\int_t^T\int_{B_{\lambda/2}(p)}|\mathcal T(u(s))|^2\,dV\,ds
		\le C\frac{R^2+\delta+\delta R^2}{\ln(R/a)}
		\le C\frac\delta{|\ln\delta|}.
		\]
		Substituting this into \eqref{eq:local-tension-scale} yields
		\[
		\lambda^2\le C\frac{\delta^2}{|\ln\delta|},
		\]
		which proves \eqref{eq:outer-scale-log-improvement}.
	\end{proof}

	\section{Tension and stress estimates on bubble regions}
	\label{sec:bubble-estimates}
	
	In this section, we show that the $L^2$ tension and stress energies do not concentrate inside bubble regions, including ghost bubbles. The starting point is the evolution equation of $\Phi(u)$ in Proposition~\ref{prop:antiholomorphic-evolution}. We define rescaled flows $v_n$ along a bubbling sequence and study the limit of normalized antiholomorphic differentials of $v_n$. By establishing the rigidity of ancient solutions to the limiting system, we show that inside the bubble region,
	\[ |\mathcal T(v_n)|+|S(v_n)|=o(\lambda_n), \]
	and the tension and stress energies vanish.
	
	\subsection{Spacetime convergence of rescaled flows}
	\label{subsec:rescaled-flow-convergence}
	
	Fix $t_n\uparrow T$, $\lambda_n\downarrow0$, and points $p_n\in\Sigma$ with coordinates $z_n=z(p_n)$ in a fixed positively oriented conformal coordinate chart $z$.  Assume that
	\begin{equation}
		u(z_n+\lambda_n y,t_n)
		\longrightarrow
		\omega(y)
		\qquad
		\text{strongly in }
		C^{\infty}_{\mathrm{loc}}(\mathbb R^2\setminus Z),
		\label{eq:rescaled-terminal-convergence}
	\end{equation}
	where $Z\subset\mathbb R^2$ is finite and $\omega$ is a smooth harmonic map on $\mathbb R^2$ with finite energy. Here, the limit map $\omega$ is allowed to be constant.
	
	Define the rescaled flows by
	\begin{equation}
		v_n(y,s)
		:=
		u(z_n+\lambda_n y,t_n+\lambda_n^2s).
		\label{eq:rescaled-flow}
	\end{equation}	
	For every compact subset of $\mathbb{R}^2\times(-\infty, 0]$, these flows are defined for all sufficiently large $n$. Let $F_n(y)=z_n+\lambda_n y$ and $g_n=\lambda_n^{-2}F_n^*g$ be the rescaled domain metric. Then these metrics converge smoothly on compact sets to a constant Euclidean metric, and $v_n$ satisfies
	\[ \partial_sv_n=\mathcal{T}_{g_n}(v_n). \]
	
	The first lemma upgrades the endpoint convergence of $v_n$ in \eqref{eq:rescaled-terminal-convergence} to spacetime convergence of $v_n$ to a stationary ancient solution.
	
	\begin{lemma}
		\label{lem:rescaled-spacetime-convergence}
		After passing to a subsequence,
		\begin{equation*}
			v_n(y,s)
			\longrightarrow
			\omega(y)
			\qquad
			\text{strongly in }
			C^{\infty}_{\mathrm{loc}}((\mathbb R^2\setminus Z)\times(-\infty,0]).
		\end{equation*}
	\end{lemma}
	
	\begin{proof}
		Fix $L>0$ and a compact set $K\Subset\mathbb R^2\setminus Z$. Choose an open set $U$ with $K\Subset U\Subset\mathbb R^2\setminus Z$, and put
		\[
		\mathcal D_n(L)
		:=\int_{-L-1}^0\int_U
		|\mathcal T_{g_n}(v_n)|^2\,dV_{g_n}\,ds.
		\]
		The scaling in \eqref{eq:rescaled-flow} gives
		\begin{equation*}
			\begin{aligned}
				\mathcal D_n(L)
				&=\int_{t_n-(L+1)\lambda_n^2}^{t_n}
				\int_{F_n(U)}|\mathcal T_g(u)|^2\,dV_g\,dt\\
				&\le\int_{t_n-(L+1)\lambda_n^2}^T
				\|\mathcal T(u(t))\|_2^2\,dt\longrightarrow0,
			\end{aligned}
		\end{equation*}
		where $F_n(U)=z_n+\lambda_n U$ and the last limit follows from \eqref{eq:energy-dissipation}.
		
		By \eqref{eq:rescaled-terminal-convergence}, each $y\in K$ admits a disk $D_R(y)\Subset U$, with $R$ sufficiently small, such that $E_{g_n}(v_n(0);D_R(y))<\varepsilon_0/2$ for all large $n$. Here $\varepsilon_0$ is the threshold in Proposition~\ref{prop:interior-parabolic-regularity}. Choose a spatial cutoff $0\le\chi\le1$ supported in $D_R(y)$, equal to one on $D_{R/2}(y)$, with $|d\chi|_{g_n}\le C/R$. Lemma~\ref{lem:local-energy-comparison}, applied with the rescaled metric, gives, for $-L-1\le s\le0$,
		\[
		E_{g_n}(v_n(s);D_{R/2}(y))
		\le E_{g_n}(v_n(0);D_R(y))
		+\mathcal D_n(L)
		+\frac{C}{R}\sqrt{(L+1)E_0\mathcal D_n(L)}
		<\varepsilon_0
		\]
		for all large $n$. The support of $\chi$ lies in $U$, so all dissipation terms in this comparison are bounded by $\mathcal D_n(L)$.
		
		The smooth convergence of $g_n$ to a constant metric makes the constants in Proposition~\ref{prop:interior-parabolic-regularity} uniform. Taking $R$ small enough that the required parabolic waiting time is less than one, that proposition gives uniform derivative bounds on $D_{R/4}(y)\times[-L,0]$. A finite covering of $K$ and a diagonal subsequence yield smooth local convergence to an ancient flow $v_\infty$ on $(\mathbb R^2\setminus Z)\times(-\infty,0]$. Moreover,
		\[
		\int_{-L}^0\int_K|\partial_s v_n|^2\,dV_{g_n}\,ds
		\le\mathcal D_n(L)\longrightarrow0.
		\]
		Passing to the limit gives $\partial_s v_\infty=0$. Since $v_\infty(\cdot,0)=\omega$ by \eqref{eq:rescaled-terminal-convergence}, the limit is $\omega$ for every $s\le0$.
	\end{proof}
	
	\subsection{Limiting system of the normalized antiholomorphic differential}
	
	By \eqref{assumption:H},
	\begin{equation*}
		|\Phi(v_n)|
		\le C B\lambda_n,
	\end{equation*}
	and the bubble $\omega$ is holomorphic, i.e. $\Phi(\omega)=0$.	We consider the normalized differential
	\begin{equation}
		\Phi_n
		:=
		\frac{\Phi(v_n)}{\lambda_n}.
		\label{eq:normalized-differential}
	\end{equation}
	Then $|\Phi_n|\le CB$. We denote a subsequential limit by $\Phi_{\infty}$.
	
	In the chosen coordinates, the coefficients of $\Phi_n$ are those of $\Phi(u)$ evaluated at $(z_n+\lambda_n y,t_n+\lambda_n^2s)$. Our goal is to show that  $\nabla_{g_n}\Phi_n\to0$ on compact regular subsets. First we derive the limiting equation of $\Phi_n$.

	\begin{lemma}
		\label{lem:normalized-limiting-system}
		After passing to a further subsequence,
		\begin{equation}
			\Phi_n\longrightarrow \Phi_{\infty}
			\quad \text{strongly in }
			C^{\infty}_{\mathrm{loc}}((\mathbb R^2\setminus Z)\times(-\infty,0]).
			\label{eq:normalized-differential-convergence}
		\end{equation}
		The limit $\Phi_{\infty}$ is bounded and satisfies
		\begin{equation}
			\nabla_s\Phi_{\infty}-\Delta_\nabla \Phi_{\infty}+\mathcal R_{\Psi(\omega)} \Phi_{\infty}=0,
			\label{eq:limiting-differential-equation}
		\end{equation}
		where $\mathcal R_{\Psi(\omega)}$ is a nonnegative self-adjoint endomorphism. In particular,
		\begin{equation*}
			\frac12(\partial_s-\Delta)|\Phi_{\infty}|^2
			=-|\nabla \Phi_{\infty}|^2-\langle\mathcal R_{\Psi(\omega)}\Phi_{\infty}, \Phi_{\infty}\rangle.
		\end{equation*}
	\end{lemma}
	
	\begin{proof}
		Rescaling \eqref{eq:antiholomorphic-evolution} and dividing by $\lambda_n$ gives
		\begin{equation}\label{eq:rescaled-differential-equation}
			(\nabla_s-\Delta_\nabla)\Phi_n
			=-\lambda_n^{2}K_\Sigma\Phi_n-\mathcal R_{\Psi(v_n)}\Phi_n+\mathcal R_{\Phi(v_n)}\Phi_n.
		\end{equation}
		Here $\Delta_\nabla$ uses $g_n$ before taking the limit and its constant limit metric afterwards; $\Delta$ denotes the scalar Laplacian. This is a homogeneous system on real bundle-valued one-forms. On every compact regular cylinder, the smooth convergence $v_n\to\omega$ by Lemma~\ref{lem:rescaled-spacetime-convergence} bounds all coefficients and their derivatives. Standard interior parabolic estimates, together with the bound $|\Phi_n|\le CB$, yield uniform bounds on all derivatives on smaller cylinders. Therefore, after passing to a subsequence, $\Phi_n$ converges smoothly locally to a limit form $\Phi_{\infty}$ on $(\mathbb R^2\setminus Z)\times(-\infty,0]$, and \eqref{eq:normalized-differential-convergence} follows.
		
		By letting $n\to \infty$ in \eqref{eq:rescaled-differential-equation}, we get 
		\[
		(\nabla_s-\Delta_\nabla)\Phi_{\infty}
		=-\mathcal R_{\Psi(\omega)}\Phi_{\infty}+\mathcal R_{\Phi(\omega)}\Phi_{\infty}.
		\]
		Since $\omega$ is holomorphic, this proves \eqref{eq:limiting-differential-equation}.
	\end{proof}
	
	Next we extend the limit $\Phi_{\infty}$ to an ancient solution across the finite set $Z$.
	
	\begin{lemma}
		\label{lem:limiting-system-removability}
		The bounded solution $\Phi_{\infty}$ to \eqref{eq:limiting-differential-equation} extends across $Z$ as a smooth solution on
		$$
		\mathbb R^2\times(-\infty,0].
		$$
	\end{lemma}
	
	\begin{proof}
		It suffices to consider one point $y_0\in Z$. Choose $R>0$ such that
		\[
		\overline{D_{2R}(y_0)}\cap Z=\{y_0\},
		\]
		and fix a smooth orthonormal bundle frame over $D_{2R}(y_0)$. We write $\mathcal R=\mathcal R_{\Psi(\omega)}$ and put
		\[
		M=\|\Phi_{\infty}\|_{L^\infty}.
		\]
		
		For $0<\varepsilon<\min\{R,1/4\}$, apply Lemma~\ref{lem:logarithmic-cutoff} to the disks centered at $y_0$ with radii $\varepsilon^2$ and $\varepsilon$, and set $\zeta_\varepsilon=1-\chi_{\varepsilon^2,\varepsilon}$. Then $\zeta_\varepsilon=0$ on $D_{\varepsilon^2}(y_0)$, $\zeta_\varepsilon=1$ outside $D_\varepsilon(y_0)$, and
		\begin{equation}
			\int_{\mathbb R^2}|\nabla\zeta_\varepsilon|^2
			\le\frac{C}{|\ln\varepsilon|}\longrightarrow0.
			\label{eq:puncture-cutoff-energy}
		\end{equation}
		
		Fix $s_0<s_1<0$. Choose $\chi\in C_c^\infty(D_{2R}(y_0))$ and a smooth temporal cutoff $\psi$ such that
		\[
		0\le\chi,\psi\le1,\qquad
		\chi\equiv1\ \text{on }D_R(y_0),\qquad
		\psi(s_0)=0,\qquad
		\psi\equiv1\ \text{on }[s_1,0].
		\]
		Write $\phi_\varepsilon=\chi\zeta_\varepsilon$. Testing the equation against $\psi^2\phi_\varepsilon^2\Phi_{\infty}$ gives
		\begin{align*}
			&\frac12\frac{d}{ds}
			\left(\psi^2\int \phi_\varepsilon^2|\Phi_{\infty}|^2\right)
			+\psi^2\int \phi_\varepsilon^2
			\left(
			|\nabla\Phi_{\infty}|^2+
			\langle\mathcal R\Phi_{\infty},\Phi_{\infty}\rangle
			\right)\\
			&\qquad=
			\psi\psi'\int \phi_\varepsilon^2|\Phi_{\infty}|^2
			-2\psi^2\int \phi_\varepsilon
			\langle\nabla\Phi_{\infty},d\phi_\varepsilon\otimes\Phi_{\infty}\rangle.
		\end{align*}
		Using
		\[
		2\phi_\varepsilon|\nabla\Phi_{\infty}|\,
		|d\phi_\varepsilon|\,|\Phi_{\infty}|
		\le
		\frac12\phi_\varepsilon^2|\nabla\Phi_{\infty}|^2
		+2|d\phi_\varepsilon|^2|\Phi_{\infty}|^2
		\]
		and $\mathcal R\ge0$, and integrating from $s_0$ to $\tau<0$, we obtain
		\begin{align*}
			&\frac12\psi(\tau)^2
			\int \phi_\varepsilon^2|\Phi_{\infty}(\tau)|^2
			+\frac12\int_{s_0}^{\tau}\int
			\psi^2\phi_\varepsilon^2|\nabla\Phi_{\infty}|^2\\
			&\qquad\le
			M^2\int_{s_0}^{\tau}\int
			\left(
			|\psi\psi'|\phi_\varepsilon^2
			+2\psi^2|d\phi_\varepsilon|^2
			\right)
			\le C M^2,
		\end{align*}
		where $C$ is independent of $\varepsilon$ and $\tau$. Indeed,
		\[
		|d\phi_\varepsilon|^2
		\le
		2\zeta_\varepsilon^2|d\chi|^2
		+2\chi^2|d\zeta_\varepsilon|^2,
		\]
		and the second term is controlled by \eqref{eq:puncture-cutoff-energy}. Letting $\tau\uparrow0$ and then $\varepsilon\downarrow0$, Fatou's lemma gives
		\[
		\int_{s_1}^{0}\int_{D_R(y_0)\setminus\{y_0\}}
		|\nabla\Phi_{\infty}|^2<\infty.
		\]
		
		Extend $\Phi_{\infty}$ measurably across $\{y_0\}\times(s_1,0)$, with arbitrary values on this measure-zero set. The sections $\zeta_\varepsilon\Phi_{\infty}$ converge to $\Phi_{\infty}$ in $L^2$, and
		\[
		\nabla(\zeta_\varepsilon\Phi_{\infty})
		=
		\zeta_\varepsilon\nabla\Phi_{\infty}
		+d\zeta_\varepsilon\otimes\Phi_{\infty}.
		\]
		Moreover,
		\[
		\|d\zeta_\varepsilon\otimes\Phi_{\infty}\|_
		{L^2(D_R(y_0)\times(s_1,0))}^2
		\le
		M^2|s_1|
		\int_{\mathbb R^2}|d\zeta_\varepsilon|^2
		\longrightarrow0.
		\]
		Together with dominated convergence for $\zeta_\varepsilon\nabla\Phi_{\infty}$, this shows that
		\[
		\Phi_{\infty}\in L^2\bigl((s_1,0);H^1(D_R(y_0))\bigr),
		\]
		and its weak covariant derivative agrees with the original derivative away from $y_0$.
		
		Let $\varphi$ be a smooth bundle-valued test section compactly supported in $D_R(y_0)\times(s_1,0)$. Testing the equation away from $y_0$ against $\zeta_\varepsilon\varphi$ yields
		\begin{align*}
			0={}&
			\int_{s_1}^{0}\int_{D_R(y_0)}
			\zeta_\varepsilon
			\left(
			-\langle\Phi_{\infty},\partial_s\varphi\rangle
			+\langle\nabla\Phi_{\infty},\nabla\varphi\rangle
			+\langle\mathcal R\Phi_{\infty},\varphi\rangle
			\right)\\
			&+
			\int_{s_1}^{0}\int_{D_R(y_0)}
			\langle\nabla\Phi_{\infty},d\zeta_\varepsilon\otimes\varphi\rangle.
		\end{align*}
		The last integral tends to zero, since its absolute value is bounded by
		\[
		\|\nabla\Phi_{\infty}\|_{L^2(D_R(y_0)\times(s_1,0))}
		\|\varphi\|_{L^\infty}
		\left(
		|s_1|\int_{\mathbb R^2}|d\zeta_\varepsilon|^2
		\right)^{1/2}.
		\]
		All remaining terms converge by dominated convergence. Hence
		\[
		\int_{s_1}^{0}\int_{D_R(y_0)}
		\left(
		-\langle\Phi_{\infty},\partial_s\varphi\rangle
		+\langle\nabla\Phi_{\infty},\nabla\varphi\rangle
		+\langle\mathcal R\Phi_{\infty},\varphi\rangle
		\right)=0.
		\]
		Thus $\Phi_{\infty}$ solves the equation weakly across $y_0$.
		
		In the chosen smooth frame, this is a uniformly parabolic linear system with smooth coefficients and scalar principal part. Interior parabolic regularity therefore gives a smooth representative on $D_R(y_0)\times(s_1,0)$, agreeing with the original solution away from $y_0$.
		
		To include the terminal face, fix $s_2\in(s_1,0)$. Local estimates on backward cylinders give, for every $k,\ell\ge0$,
		\[
		\sup_{D_{R/2}(y_0)\times[s_2,0)}
		|\nabla^k\partial_s^\ell\Phi_{\infty}|
		\le C_{k,\ell}M.
		\]
		The constants are uniform as the upper time face tends to zero, since the spatial and backward temporal margins remain fixed. Applying the same bounds to one further spacetime derivative shows that all derivatives extend continuously to $s=0$. The extension is therefore smooth up to the terminal face and satisfies the equation there. For $x\ne y_0$, its terminal value agrees with the prescribed $\Phi_{\infty}(x,0)$ by continuity.
		
		Since $s_1<0$ is arbitrary, these local extensions cover $D_{R/2}(y_0)\times(-\infty,0]$. They agree on overlaps because they agree away from $y_0$. Repeating the argument at each point of $Z$ completes the proof.
	\end{proof}
	
	\subsection{Rigidity of ancient solutions to the limiting system}
	
	For the ancient solution $\Phi_{\infty}$ to the limiting system, we first prove that it is actually stationary by local parabolic estimates, then use two-dimensional logarithmic capacity to identify the stationary fields.
	
	\begin{lemma}
		\label{lem:ancient-stationarity}
		Let $\mathcal E\to\mathbb R^2$ be a real metric vector bundle with a time-independent metric connection, and let $\mathcal R\ge0$ be a smooth time-independent self-adjoint endomorphism. Suppose $\Phi_{\infty}$ is a bounded ancient solution to
		\begin{equation}
			\partial_s\Phi_{\infty}-\Delta_\nabla \Phi_{\infty}+\mathcal R\Phi_{\infty}=0
			\label{eq:ancient-bundle-equation}
		\end{equation}
		on $\mathbb R^2\times(-\infty,0]$. Then $\Phi_{\infty}$ is stationary, i.e.
		\begin{equation*}
			\partial_s\Phi_{\infty}\equiv0.
		\end{equation*}
	\end{lemma}
	
	\begin{proof}
		Set
		\[
		M:=\|\Phi_{\infty}\|_{L^\infty(\mathbb R^2\times(-\infty,0])}.
		\]
		Fix $(x_0,s_0)\in\mathbb R^2\times(-\infty,0)$. For $R>1$, write
		\[
		D_R:=D_R(x_0),\qquad
		I_R:=[s_0-R^2,s_0],\qquad
		\Omega_R:=D_R\times I_R.
		\]
		Choose $\chi\in C_c^\infty(D_{2R})$ and a smooth temporal cutoff $\psi$ such that
		\[
		0\le\chi,\psi\le1,\qquad
		\chi\equiv1\ \text{on }D_R,\qquad
		\psi(s_0-4R^2)=0,\qquad
		\psi\equiv1\ \text{on }I_R,
		\]
		and
		\[
		|d\chi|\le C/R,\quad 0\le\psi'\le C/R^2.
		\]
		
		Testing \eqref{eq:ancient-bundle-equation} against $\psi^2\chi^2\Phi_{\infty}$ and integrating by parts in space, we obtain
		\begin{align*}
			&\frac12\frac{d}{ds}
			\left(\psi^2\int\chi^2|\Phi_{\infty}|^2\right)
			+\psi^2\int\chi^2
			\left(
			|\nabla\Phi_{\infty}|^2+\langle\mathcal R\Phi_{\infty},\Phi_{\infty}\rangle
			\right)\\
			&\qquad=
			\psi\psi'\int\chi^2|\Phi_{\infty}|^2
			-2\psi^2\int
			\chi\langle\nabla\Phi_{\infty},d\chi\otimes\Phi_{\infty}\rangle.
		\end{align*}
		The last term satisfies
		\[
		2\chi|\nabla\Phi_{\infty}|\,|d\chi|\,|\Phi_{\infty}|
		\le
		\frac12\chi^2|\nabla\Phi_{\infty}|^2
		+2|d\chi|^2|\Phi_{\infty}|^2.
		\]
		Integrating over $I_{2R}$, using the vanishing of $\psi$ at the initial time, and discarding the nonnegative terminal term gives
		
		\[	\int_{\Omega_{2R}}\psi^2\chi^2
		\left(
		\frac12|\nabla\Phi_{\infty}|^2
		+\langle\mathcal R\Phi_{\infty},\Phi_{\infty}\rangle
		\right)
		\le
		\int_{\Omega_{2R}}
		\left(
		\psi\psi'\chi^2
		+2\psi^2|d\chi|^2
		\right)|\Phi_{\infty}|^2
		\le CM^2R^2.
		\]
		It follows that
		\begin{equation}\label{eq:ancient-caccioppoli}
			\int_{\Omega_{R}}
			\left(
			|\nabla\Phi_{\infty}|^2+\langle\mathcal R\Phi_{\infty},\Phi_{\infty}\rangle
			\right)
			\le CM^2R^2.
		\end{equation}
		
		Next, set $U:=\partial_s\Phi_{\infty}$ and test \eqref{eq:ancient-bundle-equation} against $\psi^2\chi^2U$. Time independence of the metric connection implies
		\[
		\nabla U=\partial_s\nabla\Phi_{\infty}.
		\]
		Time independence and self-adjointness of $\mathcal R$ also give
		\[
		\langle\mathcal R\Phi_{\infty},U\rangle
		=\frac12\partial_s\langle\mathcal R\Phi_{\infty},\Phi_{\infty}\rangle.
		\]
		Consequently, if
		\[
		A_R(s):=
		\int\chi^2
		\left(
		|\nabla\Phi_{\infty}|^2+\langle\mathcal R\Phi_{\infty},\Phi_{\infty}\rangle
		\right),
		\]
		then
		\[
		\psi^2\int\chi^2|U|^2
		+\frac12\frac{d}{ds}\bigl(\psi^2A_R\bigr)
		=
		\psi\psi'A_R
		-2\psi^2\int
		\chi\langle\nabla\Phi_{\infty},d\chi\otimes U\rangle.
		\]
		Using
		\[
		2\chi|\nabla\Phi_{\infty}|\,|d\chi|\,|U|
		\le
		\frac12\chi^2|U|^2
		+2|d\chi|^2|\nabla\Phi_{\infty}|^2,
		\]
		and integrating over $I_{2R}$, we obtain
		\begin{align*}
			&\frac12\int_{\Omega_{2R}}
			\psi^2\chi^2|U|^2
			+\frac12 A_R(s_0)\\
			&\qquad\le
			\int_{I_{2R}}\psi\psi'A_R(s)\,ds
			+2\int_{\Omega_{2R}}
			\psi^2|d\chi|^2|\nabla\Phi_{\infty}|^2\\
			&\qquad\le
			\frac{C}{R^2}\int_{\Omega_{2R}}
			\left(
			|\nabla\Phi_{\infty}|^2+\langle\mathcal R\Phi_{\infty},\Phi_{\infty}\rangle
			\right).
		\end{align*}
		Since $A_R(s_0)\ge0$, the preceding estimate and \eqref{eq:ancient-caccioppoli} at radius $2R$ imply
		\begin{equation}
			\int_{\Omega_R}|\partial_s\Phi_{\infty}|^2
			\le
			\frac{C}{R^2}\int_{\Omega_{2R}}
			\left(
			|\nabla\Phi_{\infty}|^2+\langle\mathcal R\Phi_{\infty},\Phi_{\infty}\rangle
			\right)
			\le CM^2.
			\label{eq:ancient-time-derivative-bound}
		\end{equation}
		
		Differentiating equation \eqref{eq:ancient-bundle-equation} in time gives
		\[	\partial_sU-\Delta_\nabla U+\mathcal RU=0.\]
		Taking an inner product with $U$ yields
		\begin{equation*}
			\frac12(\partial_s-\Delta)|U|^2
			=
			-|\nabla U|^2-\langle\mathcal RU,U\rangle
			\le0.
		\end{equation*}
		Thus $|U|^2$ is a nonnegative scalar subsolution of the Euclidean heat equation. The parabolic mean-value inequality in spatial dimension two gives
		\begin{equation*}
			|U(x_0,s_0)|^2
			\le \frac{C}{R^4}\int_{\Omega_R}|U|^2.
		\end{equation*}
		Combining this with \eqref{eq:ancient-time-derivative-bound}, we obtain
		\[
		|U(x_0,s_0)|^2\le\frac{CM^2}{R^4}.
		\]
		Since the solution is ancient, we may let $R\to\infty$. Thus $U(x_0,s_0)=0$, and the point was arbitrary.
	\end{proof}
	
	Next we show that the stationary solution is parallel.
	\begin{lemma}
		\label{lem:stationary-parallelism}
		Let $\Phi_{\infty}$ be a bounded stationary solution to \eqref{eq:ancient-bundle-equation}, i.e.
		\begin{equation}
			-\Delta_\nabla \Phi_{\infty}+\mathcal R \Phi_{\infty}=0
			\qquad\text{on }\mathbb R^2.
			\label{eq:stationary-bundle-equation}
		\end{equation}
		Then $\nabla \Phi_{\infty}\equiv0$ and $\mathcal R \Phi_{\infty}\equiv0$.
	\end{lemma}
	
	\begin{proof}
		Fix $L>1$. For $R>4L$, take $\zeta=\chi_{L,R}$ from Lemma~\ref{lem:logarithmic-cutoff}. Then $\zeta=1$ on $D_L$, $\zeta$ is supported in $D_R$, and
		\[
		\int_{\mathbb R^2}|\nabla\zeta|^2
		\le\frac C{\ln(R/L)}\longrightarrow0\qquad(R\to\infty).
		\]
		Testing \eqref{eq:stationary-bundle-equation} against $\zeta^2\Phi_{\infty}$ gives
		\begin{equation*}
			\int\zeta^2|\nabla \Phi_{\infty}|^2
			+\int\zeta^2\langle\mathcal R \Phi_{\infty},\Phi_{\infty}\rangle
			=-2\int\zeta	\langle\nabla \Phi_{\infty},\Phi_{\infty}\otimes\nabla\zeta\rangle.
		\end{equation*}
		By Young's inequality,
		\begin{equation*}
			\frac12\int\zeta^2|\nabla \Phi_{\infty}|^2
			+\int\zeta^2\langle\mathcal R \Phi_{\infty},\Phi_{\infty}\rangle
			\le C\|\Phi_{\infty}\|_\infty^2\int|\nabla\zeta|^2.
		\end{equation*}
		Since $\zeta=1$ on $D_L$, letting $R\to\infty$ yields
		\begin{equation*}
			\int_{D_L}|\nabla \Phi_{\infty}|^2
			+\int_{D_L}\langle\mathcal R \Phi_{\infty},\Phi_{\infty}\rangle=0.
		\end{equation*}
		Both integrands are nonnegative, so
		\begin{equation*}
			\nabla \Phi_{\infty}=0,\qquad\langle\mathcal R \Phi_{\infty},\Phi_{\infty}\rangle=0
			\quad\text{on }D_L.
		\end{equation*}
		As $L$ is arbitrary and $\mathcal R$ is nonnegative, the lemma follows.
	\end{proof}
	
	\subsection{Vanishing of tension energy on bubble regions}
	
	Now we apply the preceding lemmas to show the vanishing of $L^2$ tension energy in the bubble region.
	
	\begin{theorem}
		\label{thm:bubble-tension-vanishing}
		Assume \eqref{assumption:H} and $B_N\ge 0$. Let $t_n\uparrow T$, $\lambda_n\downarrow0$, and coordinate centers $z_n=z(p_n)$ satisfy
		\begin{equation*}
			u(z_n+\lambda_n y,t_n)
			\longrightarrow
			\omega(y)
			\qquad
			\text{strongly in }
			C^{\infty}_{\mathrm{loc}}(\mathbb R^2\setminus Z),
		\end{equation*}
		where $Z\subset\mathbb R^2$ is finite and $\omega$ is a holomorphic bubble or constant. Then for every compact regular set $K\Subset\mathbb R^2\setminus Z$,
		\begin{equation*}
			\int_{z_n+\lambda_nK}|\mathcal T(u(t_n))|^2\,dV
			\longrightarrow0.
		\end{equation*}
	\end{theorem}
	
	\begin{proof}
		Use $v_n$ and $g_n$ from Section~\ref{subsec:rescaled-flow-convergence}, and $\Phi_n$ from \eqref{eq:normalized-differential}. By Lemmas~\ref{lem:rescaled-spacetime-convergence} and~\ref{lem:normalized-limiting-system}, $\Phi_n$ subconverges smoothly on compact regular cylinders to a bounded ancient field $\Phi_{\infty}$. Lemma~\ref{lem:limiting-system-removability} extends it across $Z$, and Lemmas~\ref{lem:ancient-stationarity} and~\ref{lem:stationary-parallelism} give
		\[
		\nabla\Phi_{\infty}=0,\qquad \mathcal R_{\Psi(\omega)}\Phi_{\infty}=0.
		\]
		The first-order identity \eqref{eq:real-first-order} therefore yields
		\begin{equation*}
			\frac{\mathcal T_{g_n}(v_n)}{\lambda_n}
			=2\operatorname{div}_{g_n}\Phi_n\longrightarrow0,
		\end{equation*}
		smoothly on compact regular cylinders, including $s=0$. Every subsequence has a further subsequence with this zero limit, so the convergence holds along the whole sequence. By Lemma~\ref{lem:tension-stress-scaling},
		\[
		\int_{z_n+\lambda_nK}|\mathcal T(u(t_n))|^2\,dV
		=\int_K\left|\frac{\mathcal T_{g_n}(v_n(\cdot,0))}{\lambda_n}\right|^2
		\,dV_{g_n}\longrightarrow0.
		\]
	\end{proof}
	
	\begin{remark}
		If the holomorphic bisectional curvature is strictly positive along the bubble, then
		$$
		\mathcal R_{\Psi(\omega)} \Phi_{\infty}=0
		$$
		forces $\Phi_{\infty}=0$ on every nonconstant component. This can also be proved using the usual absorption argument for a scalar equation.
	\end{remark}
	
	\begin{remark}
		If $\omega$ is constant, then
		$$
		\mathcal R_{\Psi(\omega)}\equiv0.
		$$
		In this case, the bounded ancient normalized differential $\Phi_{\infty}$ is not necessarily zero, but still parallel. Thus the same proof treats nonconstant bubbles and ghost components.
	\end{remark}
	
	\subsection{Vanishing of stress energy on bubble regions}
	
	We now establish the stress energy counterpart of Theorem~\ref{thm:bubble-tension-vanishing}, which will be used to prove strong stress convergence in Section~\ref{sec:terminal-convergence}. First we prove the following lemma.

	\begin{lemma}\label{lem:stress-pairing}
		Let $\omega:S^2\to N$ be harmonic and let $X$ be a bounded parallel real section of $\omega^*TN$ on an affine plane minus finitely many points. Then $X$ extends to a parallel section over $S^2$, and
		\begin{equation*}
			h(d\omega,X)=0.
		\end{equation*}
	\end{lemma}
	\begin{proof}
		We first extend the parallel section $X$ to a global one on $S^2$. In a smooth local bundle frame, parallelism of $X$ gives 
		\[ dX=-\Gamma^N X, \]
		where $\Gamma^N$ is the connection form which is smooth across the singular points. Boundedness of $X$ then gives bounded first derivatives near each missing point. Paths of length tending to zero in a punctured disk give a unique continuous extension; the parallel equation $\nabla X=0$ then extends weakly and implies smoothness. The same argument applies at infinity in the coordinate $1/z$.
		
		The global real one-form $\alpha=h(d\omega,X)$ is closed because $d^\nabla d\omega=0$ and $\nabla X=0$, and is co-closed because $\mathcal T(\omega)=0$. Since $H^1(S^2;\mathbb R)=0$, we may write $\alpha=df$. Co-closedness gives $\Delta f=0$, so $f$ is constant and $\alpha=0$.
	\end{proof}
	
	\begin{theorem}
		\label{thm:bubble-stress-vanishing}
		Under the assumptions of Theorem~\ref{thm:bubble-tension-vanishing}, for every compact regular set $K\Subset\mathbb R^2\setminus Z$,
		\begin{equation}
			\int_{z_n+\lambda_n K}|S(u(t_n))|_g^2\,dV\longrightarrow0.
			\label{eq:bubble-stress-vanishing}
		\end{equation}
	\end{theorem}
	\begin{proof}
		As in the proof of Theorem~\ref{thm:bubble-tension-vanishing}, $\Phi_n$ subconverges to a bounded parallel field $\Phi_{\infty}$, while $\Psi(v_n)\to d\omega$ smoothly on compact regular cylinders.
		
		Using the bilinear expression of the stress-energy tensor \eqref{eq:stress-expression}, we derive
		\[\begin{aligned}
			\frac{S(v_n)}{\lambda_n}&=h(\Phi_n\otimes \Psi(v_n)) + h(\Psi(v_n)\otimes \Phi_n)\\
			&\longrightarrow h(\Phi_{\infty}\otimes d\omega)+h(d\omega\otimes\Phi_{\infty})=0.
		\end{aligned}  \]
		The last equality follows from Lemma~\ref{lem:stress-pairing}, applied to each component of the parallel one-form $\Phi_{\infty}$.   Since every subsequence has the same possible limit, the convergence holds for the whole sequence.
		
		Finally, the scaling property in Lemma~\ref{lem:tension-stress-scaling} gives
		\[
		\int_{F_n(K)}|S(u(t_n))|_g^2\,dV
		=\int_K\left|\frac{S(v_n(\cdot,0))}{\lambda_n}\right|_{g_n}^2
		\,dV_{g_n}\longrightarrow0,
		\]
		which proves \eqref{eq:bubble-stress-vanishing}.
	\end{proof}
	
	\section{Tension estimates on neck regions}\label{sec:neck-estimates}
	
	In this section, we show that the tension energy does not concentrate on the neck region by establishing control of derivatives of $\Phi$.  We first work with a normalized harmonic map flow on a Euclidean neck domain and derive the tension estimate for $\Phi$ assuming the required decay of the flow. Then we deal with the error term brought by the domain metric in Corollary~\ref{cor:original-scale-neck-tension}. To keep the main context concise, we put the technical scalar estimates in Appendix~\ref{sec:annular-estimates}. 
	
	For $0<a<R/2$, and a center $p\in\Sigma$ in a fixed chart $z$, define the parabolic neck regions
	\[
	\begin{aligned}
		\Omega_a^R(p,t_*)
		&:=\left\{(x,t):t_*-R^2\le t\le t_*,\quad
		a+\frac{t_*-t}{2R}<|x-z(p)|<R\right\},\\
		\widetilde \Omega_a^R(p,t_*)
		&:=\left\{(x,t):t_*-2R^2<t\le t_*,\quad
		\frac12\left(a+\frac{t_*-t}{2R}\right)<|x-z(p)|<2R\right\}.
	\end{aligned}
	\]
	In normalized Euclidean coordinates, we use the same formulas with center $0$ and abbreviate
	\[
	\Omega_a:=\Omega_a^1(0,0),\qquad
	\widetilde \Omega_a:=\widetilde \Omega_a^1(0,0).
	\]
	
	\subsection{Tension estimates on normalized neck regions}
	
	Fix $0<\beta<1$, $L<\infty$, and $0<a<1/2$. Let $v$ be a smooth harmonic map heat flow on $\widetilde \Omega_a$, with Euclidean domain metric. Suppose
	\begin{equation}\label{eq:standard-profile}
		|\Phi(v)|\le M,\qquad
		r|dv|\le L\left[\left(\frac{a-s/2}{r}\right)^\beta+r^\beta\right],
		\qquad r=|y|.
	\end{equation}
	Using the fixed embedding $N\hookrightarrow\mathbb R^m$, we regard $\Phi(v)$ as a $\mathbb R^m$-valued function. We write $\partial$ and $\Delta$ for Euclidean derivatives of the ambient components, and $\nabla$ for the induced target connection.
	
	\begin{lemma}
		\label{lem:normalized-neck-equation}
		Under \eqref{eq:standard-profile}, the $\mathbb R^m$-valued function $\Phi(v)$ satisfies
		\begin{equation}\label{eq:flat-neck-system}
			(\partial_s-\Delta)\Phi
			=\mathsf C^j\partial_j\Phi+\mathsf V\Phi,
		\end{equation}
		where
		\begin{equation}\label{eq:neck-coefficient-bounds}
			|\mathsf C|\le C|dv|,\qquad |\mathsf V|\le C|dv|^2.
		\end{equation}
		On $\Omega_a$, we have
		\begin{equation}\label{eq:neck-derivatives}
			|\partial\Phi|\le CM/r,\qquad
			|\partial^2\Phi|\le CM/r^2,
		\end{equation}
		and
		\begin{equation}\label{eq:neck-forcing}
			|(\partial_s-\Delta)\Phi|
			\le\frac{CM}{r^2}
			\left[\left(\frac{a-s/2}{r}\right)^\beta+r^\beta\right].
		\end{equation}
		The constants depend only on $\beta,L$ and the target geometry.
	\end{lemma}
	
	\begin{proof}
		By Proposition~\ref{prop:antiholomorphic-evolution},
		\begin{equation*}
			(\nabla_s-\Delta_\nabla)\Phi
			=-\mathcal R_{\Psi(v)}\Phi+\mathcal R_{\Phi(v)}\Phi.
		\end{equation*}
		Write $\Phi=\Phi_j\,dy^j$ in the fixed embedding $N\hookrightarrow\mathbb R^m$. The Gauss formula gives
		\[
		\partial_i\Phi_j=\nabla_i\Phi_j+\mathrm{II}(v_i,\Phi_j),
		\qquad
		\partial_s\Phi_j=\nabla_s\Phi_j+\mathrm{II}(v_s,\Phi_j).
		\]
		Differentiating the first identity and taking the trace, we obtain
		\[
		\Delta\Phi_j=(\Delta_\nabla\Phi)_j
		+\mathrm{II}(\mathcal T(v),\Phi_j)
		+2\sum_i\mathrm{II}(v_i,\nabla_i\Phi_j)
		+O(|dv|^2|\Phi|),
		\]
		where the remainder is linear in $\Phi$. Since $v_s=\mathcal T(v)$, the terms containing the tension cancel in $(\partial_s-\Delta)\Phi$. Expressing $\nabla\Phi$ in terms of $\partial\Phi$ then gives \eqref{eq:flat-neck-system} and \eqref{eq:neck-coefficient-bounds}.
		
		Fix $(y,s)\in\Omega_a$ and write $r=|y|$. We choose $0<\theta<1/4$ uniformly, independently of $a,M,y,s$. Since $a-s/2<r<1$ and $s\ge-1$, the backward cylinder
		\[
		D_{\theta r}(y)\times[s-\theta^2r^2,s]
		\]
		lies in $\widetilde\Omega_a$; in particular, its inner boundary condition follows from
		\[
		a-s/2+\tfrac12\theta^2r^2<2(1-\theta)r.
		\]
		On this cylinder, \eqref{eq:standard-profile} gives $|dv|\le C_\beta L/r$. Under the rescaling $(x,t)=(y+\theta r z,s+\theta^2r^2\tau)$, the flow therefore has uniformly bounded spatial gradient. Standard interior parabolic regularity for the flow gives uniform bounds on all its spacetime derivatives on $D_{1/2}\times[-1/4,0]$. Hence the rescaled coefficients $\theta r\,\mathsf C^j$ and $\theta^2r^2\mathsf V$ in \eqref{eq:flat-neck-system} have uniform smooth bounds there. The ambient components of $\Phi$, composed with this rescaling, remain bounded by $M$. Interior estimates for the linear system, followed by scaling back, therefore give
		\[
		|\partial^k\Phi(y,s)|\le \frac{CM}{(\theta r)^k}
		\le \frac{CM}{r^k},\qquad k=1,2.
		\]
		Here $C$ depends only on $\beta,L$ and the target geometry. These are \eqref{eq:neck-derivatives}; the same backward-cylinder estimates apply when $s=0$.
		
		Finally, \eqref{eq:flat-neck-system} implies
		\[
		|(\partial_s-\Delta)\Phi|
		\le\frac{CM}{r^2}\bigl(r|dv|+r^2|dv|^2\bigr).
		\]
		Inserting the bound \eqref{eq:standard-profile} proves \eqref{eq:neck-forcing}.
	\end{proof}
	
	Next we apply the decay estimate in Appendix~\ref{sec:annular-estimates} to control the tension energy on the normalized neck region. 
	
	\begin{proposition}\label{prop:normalized-neck-tension}
		Under \eqref{eq:standard-profile}, there exist $K_0\ge4$, $0<\delta_0<1$, and $C<\infty$ such that
		\begin{equation}\label{eq:neck-main}
			\int_{U_{Ka}^{\delta}}
			\bigl(|\nabla\Phi(v(y,0))|^2+|\mathcal T(v(y,0))|^2\bigr)\,dy
			\le CM^2\left[
			\frac1{\ln K}+\frac1{\ln(1/\delta)}+\delta^{\beta/4}\right]
		\end{equation}
		for $K\ge K_0$ and $0<\delta<\delta_0$. The constants depend only on $\beta,L$ and the target geometry, and are independent of $a$ and $M$.
	\end{proposition}
	
	\begin{proof}
		By Lemma~\ref{lem:normalized-neck-equation}, each ambient component of $\Phi$ satisfies the assumptions of Theorem~\ref{thm:moving-annulus-estimate}, with the constant $L$ in that theorem equal to $CM$. Applying the theorem with $c=1/2$ and $\mu=3/4$, so that $2(\beta/2)(1-3/4)=\beta/4$, gives
		\[
		\int_{U_{Ka}^{\delta}}|\partial\Phi(y,0)|^2\,dy	
		\le CM^2\left[
		\frac1{\ln K}+\frac1{\ln(1/\delta)}
		+\delta^{\beta/4}
		\right].
		\]
		Since $|\nabla\Phi(v)|\le|\partial\Phi(v)|$, and $|\mathcal T(v)|\le C|\nabla\Phi(v)|$ by \eqref{eq:real-first-order}, the estimate \eqref{eq:neck-main} follows.
	\end{proof}
	
	\subsection{Tension estimates at original scale}
	
	We now include the domain metric and return to the original scale.
	
	\begin{corollary}\label{cor:original-scale-neck-tension}
		Fix $0<\beta<1$ and $L<\infty$. Let $0<2a<R\le R_0$. Suppose $u$ is a smooth harmonic map heat flow on $\widetilde \Omega_a^R(p,t_*)$ satisfying
		\begin{equation*}
			|\Phi(u)|_g\le B,\qquad
			r|du|_g\le L\left[
			\left(\frac{a+(t_*-t)/(2R)}{r}\right)^\beta
			+\left(\frac rR\right)^\beta\right],\qquad r=|x-z(p)|.
		\end{equation*}
		Then
		\begin{equation}\label{eq:physical-neck}
			\int_{Ka<|x-z(p)|<\delta R}|\mathcal T(u(x,t_*))|^2\,dV_g
			\le CB^2\left[
			\frac1{\ln K}+\frac1{\ln(1/\delta)}+\delta^{\beta/4}\right]
		\end{equation}
		for $K\ge K_0$ and $0<\delta<\delta_0$. The constants depend only on $\beta,L$ and the fixed geometries.
	\end{corollary}
	
	\begin{proof}
		Set
		\[
		v(y,s)=u(z(p)+R y,t_*+R^2s),\qquad
		g_R=R^{-2}F^*g,\qquad F(y)=z(p)+R y.
		\]
		Under this rescaling, $\Omega_a^R(p,t_*)$ and $\widetilde \Omega_a^R(p,t_*)$ become $\Omega_{\widehat a}$ and $\widetilde \Omega_{\widehat a}$, respectively, with $\widehat a=a/R<1/2$. For the rest of this proof, write $r=|y|$. The profile becomes \eqref{eq:standard-profile} with normalized inner radius $\widehat a$, measured in $g_R$, with $M=BR$. In these coordinates,
		\begin{equation*}
			g_R=c_R^{-1}|dy|^2,\qquad
			c_R(0)=1,\qquad \tfrac12\le c_R\le2,
		\end{equation*}
		and the fixed geometry gives
		\begin{equation}\label{eq:neck-metric-bounds}
			|c_R-1|\le CR r,\qquad
			|\partial^k c_R|\le C_kR^k\quad(k\ge1).
		\end{equation}
		
		The same Gauss-formula calculation as above gives
		\begin{equation*}
			(\partial_s-c_R\Delta)\Phi
			=\mathsf C^j\partial_j\Phi+\mathsf V\Phi,
		\end{equation*}
		now with bounds 
		\[ |\mathsf C| \le C(|dv|_{g_R}+R)\]
		and
		\[ |\mathsf V| \le C(|dv|_{g_R}^2+R|dv|_{g_R}+R^2).\]
		Local regularity for this uniformly parabolic system gives \eqref{eq:neck-derivatives} by the same argument as in Lemma~\ref{lem:normalized-neck-equation}. To use the Euclidean scalar estimate, write $(\partial_s-\Delta)\Phi=(\partial_s-c_R\Delta)\Phi
		+(c_R-1)\Delta\Phi$. By \eqref{eq:neck-metric-bounds} and \eqref{eq:neck-derivatives}, the last term is bounded by $CMR/r$. Thus all additional terms are bounded by
		\[
		CM\left(\frac R r+R|dv|_{g_R}+R^2\right)
		\le\frac{CM}{r^2}r^\beta.
		\]
		Here we used $r|dv|_{g_R}\le C$ and $R r\le r^\beta$ for $r<1$ and $R\le1$. Consequently \eqref{eq:neck-forcing} is unchanged up to its constant. The scalar argument in Proposition~\ref{prop:normalized-neck-tension} therefore applies. Passing to covariant derivatives uses the same orthogonal projection in the target; the domain connection adds at most $CM^2R^2\delta^2$, which is absorbed by $CM^2\delta^{\beta/4}$. Hence \eqref{eq:neck-main} holds with the metric $g_R$, with constants independent of $\widehat a$ and $R$.
		
		Finally, Lemma~\ref{lem:tension-stress-scaling} gives
		\[
		\begin{aligned}
			\int_{Ka<|x-z(p)|<\delta R}|\mathcal T(u(x,t_*))|^2\,dV_g
			&=R^{-2}\int_{K\widehat a<|y|<\delta}
			|\mathcal T_{g_R}(v(y,0))|^2\,dV_{g_R}\\
			&\le CB^2\left[
			\frac1{\ln K}+\frac1{\ln(1/\delta)}+\delta^{\beta/4}\right],
		\end{aligned}
		\]
		since $R^{-2}M^2=B^2$. This proves \eqref{eq:physical-neck}.
	\end{proof}
	
	\section{Strong convergence of stress and tension}\label{sec:terminal-convergence}
	
	We first combine the bubble stress estimate with the energy identity to prove strong stress convergence. For tension, we verify the hypotheses of Corollary~\ref{cor:original-scale-neck-tension} on the neck regions and combine the resulting bounds with Theorem~\ref{thm:bubble-tension-vanishing}. This gives nonconcentration of tension energy and proves Theorem~\ref{thm:terminal-convergence}. Smooth convergence on the regular body region and removal of the singularities of the limiting map then yield the conclusions of Corollary~\ref{cor:terminal-expansions}.

	\subsection{Strong convergence of stress}
	
	\begin{proposition}
		\label{prop:strong-stress}
		Under the standing assumptions,
		\begin{equation}\label{eq:strong-stress-convergence}
			S(u(t))\longrightarrow S(u_T)\quad\hbox{strongly in }L^2(\Sigma),
		\end{equation}
		and the $L^2$-stress energy does not concentrate, i.e.
		\begin{equation}\label{eq:stress-nonconcentration}
			\mathfrak s(r):=\sup_{t_0\le s<T}\sup_{p\in\Sigma}
			\|S(u(s))\|_{L^2(B_r(p))}\longrightarrow0
			\qquad(r\downarrow0).
		\end{equation}
	\end{proposition}
	
	\begin{proof}
		The body map inherits $|\Phi(u_T)|\le B$ away from the singular set $\mathcal S$. By Lemma~\ref{lem:real-stress-identities}, we have
		\begin{equation}\label{eq:stress-energy-bound}
			|S(u(t))|^2\le CB^2e(u(t)),
			\qquad
			\forall t\in [0, T].
		\end{equation}
		In particular, $S(u_T)\in L^2(\Sigma)$.
		
		Fix any sequence $t_n\uparrow T$. By Proposition~\ref{prop:song-waldron-under-bounded-differential}, after passing to a subsequence, we obtain the finite bubble-tree decomposition described in Section~\ref{subsec:tree-construction}. For fixed admissible $L,d$, Theorem~\ref{thm:bubble-stress-vanishing} gives
		\begin{equation}\label{eq:stress-bubble}
			\sum_{\alpha=1}^{m}
			\|S(u(t_n))\|_{L^2(\mathcal B_{\alpha,n}(L,d))}^2
			\longrightarrow0.
		\end{equation}
		The vanishing of energy on the neck regions in Lemma~\ref{lem:component-cover} together with \eqref{eq:stress-energy-bound} implies
		\begin{equation*}
			\lim_{d\downarrow0}\lim_{L\to\infty}\limsup_{n\to\infty}
			\sum_{\alpha=1}^{m}
			\|S(u(t_n))\|_{L^2(\mathcal N_{\alpha,n}(L,d))}^2=0.
		\end{equation*}
		On the regular body region, smooth convergence gives
		\begin{equation}\label{eq:stress-body}
			\|S(u(t_n))-S(u_T)\|_{L^2(\mathcal R(d))}^2
			\longrightarrow0
		\end{equation}
		for each fixed $d>0$.
		
		Since $\Sigma$ is covered by the bubble region, neck region and the regular body region, we have
		\begin{align*}
			\|S(u(t_n))-S(u_T)\|_{L^2(\Sigma)}^2
			\le&
			\|S(u(t_n))-S(u_T)\|_{L^2(\mathcal R(d))}^2+2\|S(u_T)\|_{L^2(\Sigma\setminus\mathcal R(d))}^2\\
			&
			+2\sum_{\alpha=1}^{m}
			\left(
			\|S(u(t_n))\|_{L^2(\mathcal B_{\alpha,n}(L,d))}^2
			+\|S(u(t_n))\|_{L^2(\mathcal N_{\alpha,n}(L,d))}^2
			\right).
		\end{align*}
		The right-hand side tends to zero by \eqref{eq:stress-bubble}--\eqref{eq:stress-body} and absolute continuity of the integral of $|S(u_T)|^2$. Since every sequence $t_n\uparrow T$ has such a subsequence with the same limit $S(u_T)$, this proves \eqref{eq:strong-stress-convergence}.
		
		Finally, for every $t_0<t_*<T$,
		\[
		\sup_{t_*\le s<T}\sup_{p\in\Sigma}
		\|S(u(s))\|_{L^2(B_r(p))}
		\le
		\sup_{t_*\le s<T}\|S(u(s))-S(u_T)\|_{L^2(\Sigma)}
		+\sup_{p\in\Sigma}\|S(u_T)\|_{L^2(B_r(p))}.
		\]
		The right-hand side tends to zero as $t_*\uparrow T$ and $r\downarrow 0$. For each fixed $t_*$, smoothness on $\Sigma\times[t_0,t_*]$ also gives
		\[
		\sup_{t_0\le s\le t_*}\sup_{p\in\Sigma}
		\|S(u(s))\|_{L^2(B_r(p))}\longrightarrow0.
		\]
		These estimates prove \eqref{eq:stress-nonconcentration}.
	\end{proof}
	
	\subsection{Nonconcentration of tension energy}
	
	We first verify the hypotheses of the neck estimate. The following proposition combines outer-scale propagation with the neck estimate of Song--Waldron~\cite{SongWaldron2024} to obtain the parabolic profile required by Corollary~\ref{cor:original-scale-neck-tension}.
	
	\begin{proposition}
		\label{prop:backward-neck-profile}
		Assume \eqref{assumption:H}, and fix $0<\beta<1$ and $0<\varepsilon<\varepsilon_\beta$ as in Proposition~\ref{prop:two-tail-neck-estimate}. There exist $R_c>0$, $d_c>0$, and $L_\beta<\infty$ with the following property. Let $0<64a<R\le R_c$, $t_*<T$, and $t_1=t_*-512R^2\ge t_0$. Suppose
		\begin{equation}\label{eq:neck-buffer-hypotheses}
			E(u(t_*);B_{16R}(p)\setminus B_a(p))\le\varepsilon/8,
			\qquad
			\int_{t_1}^T\|\mathcal T(u(s))\|_2^2\,ds\le d_c.
		\end{equation}
		Then, with $a_*=32a$, on $\widetilde \Omega_{a_*}^R(p,t_*)$ we have, for $r=|x-z(p)|$,
		\begin{equation}\label{eq:backward-neck-profile}
			r|du(x,t)|_g\le L_\beta\left[
			\left(\frac{a_*+(t_*-t)/(2R)}{r}\right)^\beta
			+\left(\frac rR\right)^\beta\right].
		\end{equation}
		The constants may depend on $\beta,\varepsilon,B,E_0$ and the geometries, but not on the center, radii, or terminal time.
	\end{proposition}
	
	\begin{proof}
		Choose a spatial cutoff $\chi$ equal to one on $B_{8R}(p)$ and supported in $B_{16R}(p)$, with $|d\chi|\le C/R$. Lemma~\ref{lem:local-energy-comparison} and \eqref{eq:neck-buffer-hypotheses} give, for $t_1\le s\le t_*$,
		\[
		\begin{aligned}
			E(u(s);B_{8R}(p))
			&\le E(u(t_*);B_a(p))+\varepsilon/8+d_c+C\sqrt{E_0d_c}\\
			&\le E(u(t_*);B_a(p))+\varepsilon/2,
		\end{aligned}
		\]
		after choosing $d_c$ small. By \eqref{eq:stress-pointwise-bound} and \eqref{assumption:H}, $\|S(u(s))\|_2\le CB\sqrt{E_0}$. Proposition~\ref{prop:outer-scale-propagation} therefore gives
		\begin{equation}\label{eq:neck-local-propagation}
			\lambda_{\varepsilon,8R,p}(s)
			\le\max\left\{4a,\frac{CB\sqrt{E_0}}\varepsilon(t_*-s)\right\}
			\le\max\left\{4a,\frac{t_*-s}{16R}\right\},
		\end{equation}
		where we decrease $R_c$ so that $CB\sqrt{E_0}R_c/\varepsilon\le1/16$.
		
		Also take $8R_c\le r_\beta$ and $d_c\le d_\beta$. Proposition~\ref{prop:two-tail-neck-estimate}, with outer radius $8R$, then applies for $t\ge t_1+256R^2=t_*-256R^2$, which contains the time interval of $\widetilde \Omega_{a_*}^R(p,t_*)$. On the region in the statement, uniform coordinate comparison gives
		\[
		d_g(x,p)\ge r/2>8a+\frac{t_*-t}{8R}
		\ge2\lambda_{\varepsilon,8R,p}(t),
		\qquad d_g(x,p)\le2r<4R.
		\]
		Substituting \eqref{eq:neck-local-propagation} into \eqref{eq:two-tail-neck-estimate} and comparing $d_g(x,p)$ with $r$ proves \eqref{eq:backward-neck-profile}.
	\end{proof}
	
	Only the uniform $L^2$ stress bound is needed here, so the proposition applies uniformly at centers varying along the bubble tree. The region $\widetilde \Omega_{a_*}^R(p,t_*)$ rescales to $\widetilde \Omega_{a_*/R}$, as required by Corollary~\ref{cor:original-scale-neck-tension}.

	Now fix any sequence $t_n\uparrow T$ and pass to the bubble-tree subsequence from Proposition~\ref{prop:song-waldron-under-bounded-differential}. For fixed $L,d$, Theorem~\ref{thm:bubble-tension-vanishing} gives
	\begin{equation}\label{eq:tree-bubble-tension}
		\int_{\mathcal B_{\alpha,n}(L,d)}
		|\mathcal T(u(t_n))|^2\,dV\longrightarrow0
		\qquad(1\le\alpha\le m).
	\end{equation}
	This includes the constant components. Together with the neck estimate, it gives the following nonconcentration statement.
	
	\begin{proposition}
		\label{prop:tension-no-concentration}
		Along the extracted subsequence,
		\begin{equation}\label{eq:tension-nonconcentration}
			\lim_{r\downarrow0}\limsup_{n\to\infty}
			\int_{\bigcup_{p\in\mathcal S}B_r(p)}
			|\mathcal T(u(t_n))|^2\,dV=0.
		\end{equation}
	\end{proposition}
	
	\begin{proof}
		Apply the bubble-tree construction in Section~\ref{subsec:tree-construction}. Choose initial parameters $L_0,d_0$ in Lemma~\ref{lem:component-cover} so that, for all large $n$,
		\[
		E(u(t_n);\mathcal N_{\alpha,n}(L_0,d_0))<\varepsilon/8
		\qquad(1\le\alpha\le m).
		\]
		Fix a sufficiently large geometric constant $A$ and set
		\begin{equation*}
			a_{\alpha,n}=A L_0\lambda_{\alpha,n},\qquad
			R_{\alpha,n}=\frac{d_0}{A}\lambda_{\pi(\alpha),n}.
		\end{equation*}
		The uniform comparison of coordinate and geodesic radii ensures
		\[
		B_{16R_{\alpha,n}}(p_{\alpha,n})
		\setminus B_{a_{\alpha,n}}(p_{\alpha,n})
		\subset\mathcal N_{\alpha,n}(L_0,d_0).
		\]
		Also $a_{\alpha,n}/R_{\alpha,n}\to0$ by the separation of bubble scales, so $64a_{\alpha,n}<R_{\alpha,n}$ for all large $n$.
		
		Choose a late time $t_\dagger\ge t_0$ so that the tension energy $\mathcal D$ from $t_\dagger$ to $T$ is below $d_c$. Decrease $d_0$ so that, for every component and all large $n$,
		\[
		R_{\alpha,n}\le R_c,\qquad
		t_n-512R_{\alpha,n}^2\ge t_\dagger.
		\]
		For a component attached to the body this follows by taking $512(d_0/A)^2<T-t_\dagger$; for the other components it follows from $R_{\alpha,n}\to0$. Thus Proposition~\ref{prop:backward-neck-profile} and Corollary~\ref{cor:original-scale-neck-tension}, with inner radius $32a_{\alpha,n}$, give
		\begin{equation}\label{eq:tree-controlled-neck}
			\int_{32K a_{\alpha,n}<|x-z_{\alpha,n}|<\delta R_{\alpha,n}}
			|\mathcal T(u(t_n))|^2\,dV
			\le CB^2\left[
			\frac1{\ln K}+\frac1{\ln(1/\delta)}+\delta^{\beta/4}\right].
		\end{equation}
		The radii here are measured in the normalized chart at $z_{\alpha,n}$, with constants uniform in $n$ and $\alpha$.
		
		Given $\epsilon>0$, choose $K$ large and $\delta$ small so that the sum of these bounds over the finitely many components is below $\epsilon$. Then choose the final covering parameters $L,d$ with
		\[
		L\ge C A K L_0,\qquad d\le\frac{\delta d_0}{CA},
		\]
		where $C$ is a sufficiently large fixed geometric constant. Each neck $\mathcal N_{\alpha,n}(L,d)$ is then contained in the annulus in \eqref{eq:tree-controlled-neck}. The portions removed from the ends of these annuli are covered by the enlarged regular component pieces, which satisfy \eqref{eq:tree-bubble-tension}. Therefore, we arrive at
		\[
		\limsup_{n\to\infty}
		\int_{\Sigma\setminus\mathcal R(d)}
		|\mathcal T(u(t_n))|^2\,dV\le\epsilon.
		\]
		Since $\epsilon$ is arbitrary, this proves \eqref{eq:tension-nonconcentration}.
	\end{proof}
	
	\subsection{Proofs of Theorem~\ref{thm:terminal-convergence} and Corollary~\ref{cor:terminal-expansions}}
	
	\begin{proof}[Proof of Theorem~\ref{thm:terminal-convergence}]
		The stress conclusion follows from Proposition~\ref{prop:strong-stress}. It remains to prove the strong convergence of the tension field.
		
		Fix any sequence $t_n\uparrow T$ and pass to the subsequence in Proposition~\ref{prop:tension-no-concentration}. This proposition and smooth convergence on regular body regions give a uniform $L^2$ tension bound along the extracted sequence. Fatou's lemma therefore gives
		\begin{equation}\label{eq:tension-uT}
		\mathcal T(u_T)\in L^2(\Sigma\setminus\mathcal S).
		\end{equation}
		
		For every sufficiently small $d>0$,
		\[
		\begin{aligned}
			&\|\mathcal T(u(t_n))-\mathcal T(u_T)\|_{L^2(\Sigma)}^2\\
			&\quad\le
			\|\mathcal T(u(t_n))-\mathcal T(u_T)\|_{L^2(\mathcal R(d))}^2
			+2\int_{\Sigma\setminus\mathcal R(d)}
			\bigl(|\mathcal T(u(t_n))|^2+|\mathcal T(u_T)|^2\bigr)\,dV.
		\end{aligned}
		\]
		Taking $\limsup_{n\to\infty}$ and then $d\downarrow0$, the right-hand side tends to zero by smooth convergence, \eqref{eq:tension-nonconcentration}, and absolute continuity of the integral of $|\mathcal T(u_T)|^2$. Thus the tensions converge strongly along the extracted subsequence. Since every sequence $t_n\uparrow T$ has such a subsequence with the same limit, the convergence holds as $t\uparrow T$.
		
	\end{proof}
	
	Now Corollary~\ref{cor:terminal-expansions} follows directly.
	
	\begin{proof}[Proof of Corollary~\ref{cor:terminal-expansions}]
		The body map $u_T$ is smooth away from $\mathcal S$, has finite energy and satisfies \eqref{eq:tension-uT}. Then a standard removable-singularity theorem (cf. Li--Wang~\cite[Theorem~1.1]{LiWang2010ijm}), applied in isothermal coordinates near each point of $\mathcal S$, gives $u_T\in W^{2,2}(\Sigma,N)$. 
		
		Finally, the energy identity and strong tension convergence give
		\[
		E(u(t))-E_*=\int_t^T\|\mathcal T(u(s))\|_2^2\,ds
		=(T-t)\|\mathcal T(u_T)\|_2^2+o(T-t).
		\]
	\end{proof}

	\section{Explicit antiholomorphic energy gap and its application}\label{sec:antiholomorphic-energy-gap}

	In this section, we prove the antiholomorphic energy gap Theorem~\ref{thm:antiholomorphic-energy-threshold} and Corollary~\ref{cor:gap-flow-app}.	
	Throughout this section, $N$ is compact K\"ahler with nonnegative holomorphic bisectional curvature and 
	\begin{equation}\label{eq:gap-setting}
		H_*:=\max_{q\in N,\,0\ne X\in T_qN}H_N(X) >0,
		\qquad \delta_N:=\frac{4\pi}{H_*},
	\end{equation}
	
	\subsection{Explicit antiholomorphic gap theorem}
	
	We prove a more general gap theorem which applies to a closed Riemann surface by applying the Bochner formula to the conformal metric $e_{\bar\partial}g$ and then using Gauss--Bonnet. First we characterize the zeros of a bundle-valued harmonic 1-form.
	
	\begin{lemma}\label{lem:harmonic-one-form-zeros}
		Let $\mathcal{E}$ be a real vector bundle over a closed Riemann surface $\Sigma$, endowed with a bundle metric and compatible connection $\nabla$. Suppose $\xi$ is a nonzero harmonic $\mathcal{E}$-valued 1-form. Then the zero set $Z$ of $\xi$ is finite, and each zero $p\in Z$ has a positive integral order $m_p$. More precisely, near $p\in Z$ there is a holomorphic coordinate $z$ centered at $p$ and a smooth positive function $f$ such that
		\begin{equation}\label{eq:harmonic-one-form-zero-order}
			|\xi|=|z|^{m_p}f(z).
		\end{equation}
		
		Moreover, the conformal metric $\tilde{g}=\tfrac12|\xi|^2g$ on $\Sigma\setminus Z$ satisfies the generalized Gauss--Bonnet formula
		\begin{equation}\label{eq:gap-conformal-gauss-bonnet}
			\int_{\Sigma\setminus Z}K_{\tilde{g}}\,dV_{\tilde{g}}
			=2\pi\left(\chi(\Sigma)+\sum_{p\in Z}m_p\right).
		\end{equation}
	\end{lemma}
	\begin{proof}
		In a holomorphic coordinate $z=x+iy$, write
		\[
		g=e^{2\sigma}(dx^2+dy^2),\qquad
		\xi=\xi_x\,dx+\xi_y\,dy.
		\]
		The harmonicity of $\xi$ gives $d^\nabla\xi=0$ and $(d^\nabla)^*\xi=0$, which implies
		\[
		\nabla_x\xi_y-\nabla_y\xi_x=0,\qquad
		\nabla_x\xi_x+\nabla_y\xi_y=0.
		\]
		
		Let $\mathcal E^{\mathbb C}$ be the complex-linear extension of $\mathcal E$. The operator $\nabla^{0,1}$ defines a holomorphic structure on $\mathcal E^{\mathbb C}$: its square is the $(0,2)$ part of the curvature, which vanishes on a Riemann surface. Setting $\xi_z=\xi(\partial_z)=\tfrac12(\xi_x-i\xi_y)$, we have
		\[
		\nabla_{\bar z}\xi_z
		=\tfrac14\bigl(\nabla_x\xi_x+\nabla_y\xi_y
		+i(\nabla_y\xi_x-\nabla_x\xi_y)\bigr)=0.
		\]
		Consequently, $\xi_z$ is a holomorphic section with only discrete zeros of finite order. Since $\xi$ is real, its zeros are exactly those of $\xi_z dz$.
		
		More precisely, near any zero point $p\in Z$, we may decompose $\xi_z$ in a local holomorphic frame of $\mathcal E^{\mathbb C}$ with holomorphic coefficients. Then the finite order $m_p\ge1$ is given by the minimum of the orders of the nonzero coefficients. Factoring out this power gives
		\[
		\xi_z=z^{m_p}\eta(z),\qquad \eta(0)\ne0,
		\]
		where $\eta$ is a holomorphic section in this frame. Using the Hermitian metric on $\mathcal E^{\mathbb C}$ induced by the real bundle metric, we have
		\[
		|\xi|^2=4e^{-2\sigma}|\xi_z|^2
		=|z|^{2m_p}\,4e^{-2\sigma}|\eta(z)|^2.
		\]
		After shrinking the coordinate disk, $\eta$ is nowhere zero. Hence $f:=2e^{-\sigma}|\eta|$ is smooth and positive, proving \eqref{eq:harmonic-one-form-zero-order}. An invertible holomorphic change of frame or coordinate preserves $m_p$. Compactness of $\Sigma$ implies that $Z$ is finite.
		
		On $\Sigma\setminus Z$, put $\varphi=\ln(|\xi|/\sqrt2)$, so that $\tilde g=e^{2\varphi}g$. The conformal curvature formula gives
		\begin{equation}\label{eq:7-1}
			K_{\tilde g}\,dV_{\tilde g}
			=(K_g-\Delta_g\varphi)\,dV_g.
		\end{equation}
		To compute its integral, we remove disjoint coordinate disks $D_\varepsilon(p)=\{|z|<\varepsilon\}$ and set $\Sigma_\varepsilon=
		\Sigma\setminus\bigcup_{p\in Z}D_\varepsilon(p)$. In each disk $D_\varepsilon(p)$, the zero-order formula yields $\varphi=m_p\ln|z|+\ln(f/\sqrt2)$. On its boundary circle $\partial D_\varepsilon(p)$, we have
		\begin{equation}\label{eq:7-2}
			\int_{\partial D_\varepsilon(p)}\partial_\nu\varphi\,ds_g
			=-\int_0^{2\pi}
			\bigl(m_p+\varepsilon\partial_r\ln f\bigr)\,d\theta
			=-2\pi m_p+O(\varepsilon).
		\end{equation}
		Integrating \eqref{eq:7-1} and applying the divergence theorem now gives
		\[
		\int_{\Sigma_\varepsilon}K_{\tilde g}\,dV_{\tilde g}
		=\int_{\Sigma_\varepsilon}K_g\,dV_g
		-\sum_{p\in Z}
		\int_{\partial D_\varepsilon(p)}\partial_\nu\varphi\,ds_g.
		\]
		Letting $\varepsilon\downarrow0$, Gauss--Bonnet and \eqref{eq:7-2} together yield \eqref{eq:gap-conformal-gauss-bonnet}.
	\end{proof}
	
	Now we prove the antiholomorphic energy gap theorem for a general Riemann surface, from which Theorem~\ref{thm:antiholomorphic-energy-threshold} follows directly. Kokarev~\cite{Kokarev2013} established curvature-integral bounds for harmonic maps into K\"ahler manifolds, which actually imply the antiholomorphic energy bound recorded here. Here we give a direct proof which easily identifies the equality conditions.
	
	\begin{theorem}\label{thm:antiholomorphic-gap-general}
		Let $N$ be a compact K\"ahler manifold with nonnegative holomorphic bisectional curvature and $H_*>0$. Let $\Sigma$ be a closed Riemann surface. Every harmonic map $\omega:\Sigma\to N$ is either holomorphic or satisfies
		\begin{equation}\label{eq:antiholomorphic-zero-order-gap}
			E_{\bar\partial}(\omega)\ge \frac{2\pi}{H_*}\left(\chi(\Sigma)+\sum_p\ord_p(\Phi(\omega))\right).
		\end{equation}
	\end{theorem}
	\begin{proof}
		Let $\mathcal{E}=\omega^*TN$ be the pull-back bundle. Let $\Phi(\omega)$ be the real antiholomorphic differential of $\omega$ and $Z$ be its zero set. By Lemma~\ref{lem:real-first-order-identities}, $\Phi$ is a harmonic $\mathcal{E}$-valued 1-form.
		
		On $\Sigma\setminus Z$, set $\tilde{g}=\tfrac12|\Phi|^2g$. Then
		\[
		e_{\bar{\partial},\tilde{g}}=\frac12|\Phi|^2_{\tilde{g}}=1, \quad \operatorname{Area}_{\tilde{g}}(\Sigma\setminus Z)=E_{\bar\partial}(\omega).
		\]
		Harmonicity is conformally invariant in dimension two, so $\Phi$ is still harmonic under the conformal metric $\tilde{g}$. The stationary form of the Bochner formula \eqref{eq:antiholomorphic-bochner}, applied with the domain metric $\tilde{g}$, gives
		\begin{equation}\label{eq:gap-conformal-bochner}
			\begin{aligned}
				0=\Delta_{\tilde{g}}e_{\bar{\partial},\tilde{g}}
				&=\tfrac12|\nabla_{\tilde{g}}\Phi|_{\tilde{g}}^2+K_{\tilde{g}} +B_N(X,Y)e_{\partial, \tilde{g}}-H_N(X)
				\\
				&\ge K_{\tilde{g}}-H_*.
			\end{aligned}
		\end{equation}
		Here we used the nonnegative bisectional curvature assumption $B_N\ge 0$ in the last line. Now integrating this inequality over $\Sigma\setminus Z$ gives
		\[ \int_{\Sigma\setminus Z}K_{\tilde{g}}dV_{\tilde{g}} \le H_*\operatorname{Area}_{\tilde{g}}(\Sigma\setminus Z)=H_*E_{\bar\partial}(\omega).\]
		Applying \eqref{eq:gap-conformal-gauss-bonnet} in Lemma~\ref{lem:harmonic-one-form-zeros}, we get \eqref{eq:antiholomorphic-zero-order-gap}.
	\end{proof}
	
	\begin{remark}\label{rem:gap-equality}
		From \eqref{eq:gap-conformal-bochner}, we see that the equality in \eqref{eq:antiholomorphic-zero-order-gap} holds precisely when 
		\[ K_{\tilde g}=H_*, \]
		or equivalently,
		\begin{equation}\label{eq:gap-rigidity-kahler-equality}
			\nabla_{\tilde g}\Phi=0,\qquad
			B_N(X,Y)e_{\partial,\tilde g}=0,\qquad H_N(X)=H_*.
		\end{equation}
		In Appendix~\ref{sec:gap-rigidity}, we show that if the holomorphic bisectional curvature of $N$ is strictly positive, then the equality holds if and only if $\omega$ is an antiholomorphic totally geodesic branched immersion whose induced metric has constant Gauss curvature $H_*$ away from the branch points.
	\end{remark}
	
	\begin{remark}\label{rem:kato-equality}
		A direct computation shows that
		\[ |\nabla_{\tilde{g}}\Phi|_{\tilde{g}}^2=4|\Phi|_g^{-4}\left(|\nabla \Phi|_g^2-2|\nabla|\Phi|_g|^2\right).\]
		This gives a simple proof of the improved Kato inequality for harmonic 1-forms (cf. \cite{CibotaruZhu2012}), i.e.
		\[ |\nabla \Phi|_g^2\ge 2|\nabla|\Phi|_g|^2. \]
		In particular, the equality holds if and only if $\nabla_{\tilde{g}}\Phi=0$.
	\end{remark}
	
	Now we give the proof of Theorem~\ref{thm:antiholomorphic-energy-threshold}.
	
	\begin{proof}[Proof of Theorem~\ref{thm:antiholomorphic-energy-threshold}]
		For $\Sigma=S^2$, the Euler characteristic number is $\chi(\Sigma)=2$. If $\omega$ is nonholomorphic, then \eqref{eq:thm-gap} follows by discarding the nonnegative orders in \eqref{eq:antiholomorphic-zero-order-gap}. 
		
		Otherwise $\omega$ is holomorphic. Composing $\omega$ with an orientation-reversing conformal automorphism of $S^2$ gives an antiholomorphic map $\tilde{\omega}$. Then \eqref{eq:thm-holo-gap} follows by applying \eqref{eq:thm-gap} to $\tilde{\omega}$.
	\end{proof}
	
	\subsection{Threshold of almost-holomorphic harmonic map flow}
	
	We next use the gap theorem to prove Corollary~\ref{cor:gap-flow-app}. The key point is to rule out loss of antiholomorphic energy at the terminal time, so that assumption \eqref{assumption:H} holds. The proof follows from the split energy identity, which was proved by Li--Song \cite[Section~6.1]{LiSong2025} for the Landau--Lifshitz flow.
	
	\begin{proposition}\label{prop:antiholomorphic-bound-below-gap}
		Let $u:\Sigma\times[0,T)\to N$ be a smooth harmonic map heat flow with $T<\infty$, satisfying \eqref{eq:initial-gap-threshold}. Then \eqref{assumption:H} holds. Moreover,
		\[
		E_{\bar\partial}(u(t))\longrightarrow E_{\bar\partial}(u_T),
		\qquad
		\lim_{r\downarrow0}\lim_{t\uparrow T}
		E_{\bar\partial}(u(t);B_r(p))=0
		\]
		for every $p\in\mathcal S$.
	\end{proposition}
	\begin{proof}
		Fix $p\in\mathcal S$. The classical bubbling and energy identity \cite{LinWang1998,DingTian1995,Wang1996}, which do not require \eqref{assumption:H}, provide a sequence $t_n\uparrow T$ and finitely many harmonic spheres $\omega_{p,1},\ldots,\omega_{p,\ell_p}$ at $p$ such that
		\[
		\lim_{r\downarrow 0}\lim_{n\to\infty}E(u(t_n);B_r(p))
		=\sum_{j=1}^{\ell_p}E(\omega_{p,j}).
		\]
		Since the energy on the neck region vanishes in the limit, the energy identity can be upgraded to a split energy identity for holomorphic and antiholomorphic energies. Indeed, the strong convergence of $u(t)$ on the body regular region and the rescaled maps on the bubble regions imply
		\begin{equation}\label{eq:gap-local-split-energy}
			\lim_{r\downarrow 0}\lim_{n\to\infty}E_{\bar\partial}(u(t_n);B_r(p))
			=\sum_{j=1}^{\ell_p}E_{\bar\partial}(\omega_{p,j}).
		\end{equation}
		
		By \eqref{eq:energy-dissipation} and \eqref{eq:initial-gap-threshold},
		\[
		E_{\bar\partial}(\omega_{p,j})
		\le \lim_{n\to\infty}E_{\bar\partial}(u(t_n);B_r(p))
		\le E_{\bar\partial}(u(0))<\delta_N.
		\]
		Then Theorem~\ref{thm:antiholomorphic-energy-threshold} asserts that every $\omega_{p,j}$ is holomorphic. Hence the sum in \eqref{eq:gap-local-split-energy} vanishes, and
		\begin{equation}\label{eq:gap-local-split-energy1}
			\lim_{r\downarrow 0}\lim_{n\to\infty}E_{\bar\partial}(u(t_n);B_r(p))
			=0.
		\end{equation}
		
		The limit in \eqref{eq:gap-local-split-energy1} actually does not depend on the choice of the time sequence. Indeed, we can choose $r>0$ so small that $B_{2r}(p)\cap\mathcal S=\{p\}$. Let $\chi\in C_c^\infty(B_{2r}(p))$ satisfy $0\le\chi\le1$, $\chi=1$ on $B_r(p)$, and $|\nabla\chi|\le C/r$. Integration by parts gives
		\[
		\frac{d}{dt}\int_\Sigma\chi e_{\bar\partial}(u(t))\,dV
		=-\frac12\int_\Sigma\chi|\mathcal T|^2\,dV
		-\int_\Sigma\langle\mathcal T,\Phi(\nabla\chi)\rangle\,dV.
		\]
		Writing $D(t_1,t_2)=\int_{t_1}^{t_2}\|\mathcal T(u(t))\|_2^2\,dt$, we obtain, for $t_1<t_2<T$,
		\[
		\left|\int_\Sigma\chi
		\bigl(e_{\bar\partial}(u(t_2))-e_{\bar\partial}(u(t_1))\bigr)\,dV\right|
		\le\frac12D(t_1,t_2)
		+\frac{C\sqrt{E(u(0))}}{r}
		\sqrt{(t_2-t_1)D(t_1,t_2)}.
		\]
		Since $D(t_1,t_2)\to0$ as $t_1,t_2\uparrow T$, the cutoff energy has a limit. Therefore, using \eqref{eq:gap-local-split-energy1}, we get
		\[
		\lim_{r\downarrow 0}\lim_{t\uparrow T}E_{\bar\partial}(u(t);B_r(p))
		=\lim_{r\downarrow 0}\lim_{n\to\infty}E_{\bar\partial}(u(t_n);B_r(p))
		=0.
		\]
		Strong convergence on the regular body region also gives $E_{\bar\partial}(u(t))\to E_{\bar\partial}(u_T)$.
		
		Finally, nonconcentration of the antiholomorphic energy allows us to apply Proposition~\ref{prop:antiholomorphic-regularity} locally in a parabolic region of uniform size at every point of $\Sigma\times [0, T]$, which validates assumption \eqref{assumption:H}.
	\end{proof}
	
	\begin{proof}[Proof of Corollary~\ref{cor:gap-flow-app}]
		
		Under \eqref{eq:initial-gap-threshold}, Proposition~\ref{prop:antiholomorphic-bound-below-gap} gives \eqref{assumption:H}, which allows us to apply Theorems~\ref{thm:outer-scale-estimates} and~\ref{thm:terminal-convergence} and Corollary~\ref{cor:terminal-expansions}. This proves Corollary~\ref{cor:gap-flow-app}. 
	\end{proof}	
	
	\appendix
	
	\section{Scalar estimates on a standard neck region}
	\label{sec:annular-estimates}
	
	In this appendix, we prove scalar estimates on a parabolic region with moving inner boundary. These estimates are applied componentwise to $\Phi$ in Section~\ref{sec:neck-estimates}.  
	
	Write $r=|x|$ and $U_a^R:=\{x\in\mathbb R^2|a<r<R\}$. Fix $c>0$ and set
	\begin{equation*}
		\Omega_c(a,1)
		:=
		\left\{
		(x,s)\in \mathbb R^2\times\mathbb R^1 |
		-1\le s\le0,\quad
		a-cs < r < 1
		\right\}.
	\end{equation*}
	For $c=1/2$, this is precisely the region $\Omega_a$ used in Section~\ref{sec:neck-estimates}.	The following theorem allows any exponent $1/2<\mu<1$; the application in Proposition~\ref{prop:normalized-neck-tension} uses $\mu=3/4$.
	
	\begin{theorem}
		\label{thm:moving-annulus-estimate}
		Fix
		\[
		c>0,
		\qquad
		\beta>0,
		\qquad
		\frac12<\mu<1.
		\]
		Suppose that $f$ is smooth on $\Omega_c(a,1)$,
		\[
		|f|\le M,
		\]
		and
		\begin{equation}
			|(\partial_s-\Delta)f|
			\le
			\frac L{r^2}
			\left[
			\left(
			\frac{a-cs}r
			\right)^\beta
			+
			r^\beta
			\right].
			\label{eq:moving-annulus-forcing}
		\end{equation}
		Set
		\[
		\gamma
		:=
		\min\left\{
		\frac\beta2,\frac12
		\right\}.
		\]
		Then there exist
		\[
		K_0\ge4,
		\qquad
		\delta_0>0,
		\qquad
		C<\infty,
		\]
		depending only on $c,\beta,\mu$, such that, for
		\[
		K\ge K_0,
		\qquad
		0<\delta<\delta_0,
		\]
		we have
		\begin{equation}
			\int_{U_{Ka}^{\delta}}
			|\nabla f(x,0)|^2\,dx
			\le
			C(M+L)^2
			\left[
			\frac1{\ln K}
			+
			\frac1{\ln(1/\delta)}
			+
			\delta^{\,2\gamma(1-\mu)}
			\right].
			\label{eq:moving-annulus-energy-estimate}
		\end{equation}
	\end{theorem}
	
	First we prove the gradient estimate for a caloric function. We then include the inhomogeneous  term and obtain an estimate of the same logarithmic order. Finally, we treat the moving inner boundary and finish the proof in Section~\ref{subsec:moving-annulus-estimate}.
	
	\subsection{Neck estimate for homogeneous heat equation}
	
	We first prove a gradient estimate for bounded caloric functions on a fixed annulus. The proof uses a standard reflection argument and requires no prescribed boundary data.
	
	\begin{lemma}
		\label{lem:homogeneous-annulus-estimate}
		Let $f$ be a smooth caloric function on
		\begin{equation*}
			\Omega(a,R)
			:=
			U_a^R\times[-R^2,0],
			\qquad
			0<a<\frac{R}{16},
		\end{equation*}
		and suppose
		\[
		|f|\le M.
		\]
		Then, for all $x\in U_{4a}^{R/4}$, we have
		\begin{equation}
			|\nabla f(x,0)|
			\le
			CM
			\left[
			\frac{1}{r\ln(r/a)}
			+
			\frac1R
			\right].
			\label{eq:homogeneous-annulus-gradient}
		\end{equation}
		Consequently, for $K\ge4$ and $0<\delta\le1/4$,
		\begin{equation}
			\int_{U_{Ka}^{\delta R}}
			|\nabla f(x,0)|^2\,dx
			\le
			CM^2
			\left[
			\frac1{\ln K}
			+
			\delta^2
			\right].
			\label{eq:homogeneous-annulus-energy}
		\end{equation}
	\end{lemma}
	
	\begin{proof}
		Fix $x \in U_{4a}^{R/4}$ with $|x|=r$. Choose a unit vector $e$ such that
		\begin{equation*}
			x\cdot e=\frac{r}{\sqrt2}=:d.
		\end{equation*}
		There are two perpendicular choices of $e$, so it suffices to estimate the derivative of $f$ in either such direction.
		
		Our first step is to turn this interior derivative into a normal derivative at a zero boundary. Let
		\[
		\Pi=\{y:y\cdot e=d\},
		\]
		and denote reflection across $\Pi$ by
		\begin{equation*}
			y^*=y+2(d-y\cdot e)e.
		\end{equation*}
		Set $L=R/2$ and consider the half-disk with the inner hole removed
		\begin{equation*}
			\mathcal H
			:=\{y\cdot e<d\}\cap D_L(x)\setminus\overline{D_a}.
		\end{equation*}
		For $y\in\mathcal H$, both $y$ and $y^*$ lie in $U_a^R$: reflection fixes $x$, so
		\[
		|y|,\ |y^*|<r+L<R,
		\qquad
		|y^*|\ge y^*\cdot e>d>a.
		\]
		Consequently,
		\begin{equation*}
			g(y,t):=f(y,t)-f(y^*,t)
		\end{equation*}
		is caloric on $\mathcal H\times[-R^2,0]$, satisfies $|g|\le2M$, and vanishes on $\Pi$.
		
		We construct three nonnegative caloric barriers, one for each of the inner spatial boundary, the outer spatial boundary, and the initial time face. All three vanish on $\Pi$. Their inward normal derivatives at $(x,0)$ will therefore control that of $g$.
		
		\medskip
		\noindent
		\emph{The inner boundary.} In two dimensions, the logarithmic potential is harmonic away from its pole. Since $0$ and $2de$ are reflections across $\Pi$, the difference
		\[
		V(y):=\ln|y-2de|-\ln|y|
		\]
		vanishes on $\Pi$. Moreover,
		\[
		|y-2de|^2-|y|^2=4d(d-y\cdot e)>0
		\qquad (y\cdot e<d),
		\]
		so $V$ is positive and harmonic in $\mathcal H$. On $|y|=a$, the bound $|y-2de|\ge2d-a$ suggests the normalization
		\begin{equation*}
			H_{\mathrm{in}}(y)
			:=
			\frac{\ln\bigl(|y-2de|/|y|\bigr)}
			{\ln\bigl((2d-a)/a\bigr)}.
		\end{equation*}
		Thus $H_{\mathrm{in}}\ge1$ on the inner circle. At $x$, the inward normal is $-e$, and
		\begin{equation*}
			\partial_{-e}V(x)=\frac{2d}{r^2}
			=\frac{\sqrt2}{r}.
		\end{equation*}
		Since $2d-a\ge r$, we obtain
		\begin{equation*}
			\partial_{-e}H_{\mathrm{in}}(x)
			\le\frac{\sqrt2}{r\ln(r/a)}.
		\end{equation*}
		
		\medskip
		\noindent
		\emph{The outer boundary.} Introduce the complex coordinate
		\begin{equation*}
			z:=(y-x)\cdot e^\perp+i(d-y\cdot e).
		\end{equation*}
		The half-disk becomes $\{|z|<L,\ \operatorname{Im}z>0\}$. We seek a harmonic function equal to zero on its diameter and one on its semicircular boundary. The M\"obius transformation
		\[
		W(z)=\frac{L+z}{L-z}
		\]
		maps this half-disk to the first quadrant, sending the diameter to the positive real axis and the semicircle to the positive imaginary axis. Hence the harmonic function
		\begin{equation*}
			H_{\mathrm{out}}(y)
			:=\frac2\pi\arg\left(\frac{L+z}{L-z}\right)
		\end{equation*}
		has precisely the required boundary values.
		
		Along the inward normal ray $y=x-d_\perp e$, we have $z=id_\perp$ and $H_{\mathrm{out}}=(4/\pi)\arctan(d_\perp/L)$. Thus
		\begin{equation*}
			\partial_{-e}H_{\mathrm{out}}(x)=\frac4{\pi L}.
		\end{equation*}
		
		\medskip
		\noindent
		\emph{The initial time boundary.} Write
		\[
		d_\perp=d-y\cdot e,\qquad \tau=t+R^2.
		\]
		To dominate the initial data while retaining zero boundary value on $\Pi$, we use the half-line heat solution with initial value one and boundary value zero. The similarity ansatz $F(d_\perp/(2\sqrt\tau))$ reduces the heat equation to
		\[
		F''(\eta)+2\eta F'(\eta)=0,
		\qquad F(0)=0,\qquad F(\infty)=1.
		\]
		Its solution is the so-called  error function
		\[
		\operatorname{erf}(\eta)
		:=\frac2{\sqrt\pi}\int_0^\eta e^{-q^2}\,dq.
		\]
		Accordingly, set
		\begin{equation*}
			P(y,t)
			:=\operatorname{erf}
			\left(\frac{d-y\cdot e}{2\sqrt{t+R^2}}\right),
			\qquad -R^2<t\le0.
		\end{equation*}
		This function is nonnegative and caloric, vanishes on $\Pi$, and tends to one as $t\downarrow-R^2$ at every interior point. At the terminal time,
		\begin{equation*}
			\partial_{-e}P(x,0)=\frac1{\sqrt\pi R}.
		\end{equation*}
		
		\medskip
		The sum of these barriers dominates $|g|/(2M)$ on the parabolic boundary: $H_{\mathrm{in}}\ge1$ on the inner circle, $H_{\mathrm{out}}=1$ on the outer semicircle, and $P$ has initial value one in the interior. On $\Pi$, both $g$ and all three barriers vanish. Applying the parabolic maximum principle to $g$ and $-g$ gives
		\begin{equation}
			|g(y,t)|
			\le2M\bigl(H_{\mathrm{in}}(y)
			+H_{\mathrm{out}}(y)+P(y,t)\bigr).
			\label{eq:annular-comparison}
		\end{equation}
		The corner points can be handled by excising small neighborhoods, adding arbitrarily small logarithmic barriers, and then shrinking the neighborhoods to zero.
		
		All terms in \eqref{eq:annular-comparison} vanish at $x$. Dividing along the inward normal by the distance to $\Pi$ and taking the limit, we obtain
		\begin{equation*}
			|\partial_{-e}g(x,0)|
			\le CM\left[
			\frac1{r\ln(r/a)}+\frac1R
			\right].
		\end{equation*}
		Reflection reverses the normal direction, so
		\[
		\partial_{-e}g(x,0)=2\partial_{-e}f(x,0).
		\]
		Using the two perpendicular choices of $e$ proves \eqref{eq:homogeneous-annulus-gradient}.
		
		Finally, square this estimate and integrate in polar coordinates. For a nonempty annulus,
		\begin{equation*}
			\int_{Ka}^{\delta R}\frac{dr}{r\ln^2(r/a)}
			=\frac1{\ln K}-\frac1{\ln(\delta R/a)}
			\le\frac1{\ln K},
		\end{equation*}
		while the $R^{-2}$ term contributes at most $C\delta^2$. This proves \eqref{eq:homogeneous-annulus-energy}.
	\end{proof}

	\subsection{Neck estimate for inhomogeneous heat equation}
	
	Next we treat the inhomogeneous forcing term which decays away from the two ends of the annulus.
	
	\begin{lemma}
		\label{lem:two-tail-annulus-estimate}
		Let $f$ be a smooth solution to the equation
		\begin{equation*}
			(\partial_t-\Delta)f=G, \quad \text{on~} \Omega(a,R).
		\end{equation*}
		Suppose
		\[
		|f|\le M,
		\]
		and assume
		\begin{equation}
			|G|
			\le
			\frac{L}{r^2}
			\left[
			\left(\frac ar\right)^\beta
			+
			\left(\frac rR\right)^\beta
			\right],
			\qquad
			\beta>0.
			\label{eq:two-tail-forcing}
		\end{equation}
		Set
		\begin{equation}
			\gamma
			:=
			\min\left\{
			\frac\beta2,\frac12
			\right\}.
			\label{eq:two-tail-exponent}
		\end{equation}
		Then there exists $C_\beta<\infty$ such that, whenever
		\[
		32a\le r:=|x|\le\frac{R}{64},
		\]
		we have
		\begin{equation}
			|\nabla f(x,0)|
			\le
			C_\beta(M+L)
			\left[
			\frac1{r\ln(r/a)}
			+
			\frac1r
			\left(\frac rR\right)^\gamma
			\right].
			\label{eq:two-tail-annulus-gradient}
		\end{equation}
	\end{lemma}
	\begin{proof}
		Let $\partial\Omega(a,R)$ denote the parabolic boundary of $\Omega(a,R)$. We write $f=h+g$, where $h$ satisfies
		\begin{equation}
			\begin{cases}
				(\partial_t-\Delta)h=0,
				& \text{in } \Omega(a,R),\\
				h=f,
				& \text{on } \partial\Omega(a,R),
			\end{cases}
			\label{eq:two-tail-caloric-part}
		\end{equation}
		and $g$ satisfies
		\begin{equation}
			\begin{cases}
				(\partial_t-\Delta)g=G,
				& \text{in } \Omega(a,R),\\
				g=0,
				& \text{on } \partial\Omega(a,R).
			\end{cases}
			\label{eq:two-tail-inhomogeneous-part}
		\end{equation}
		The maximum principle gives
		\begin{equation}
			|h|\le M.
			\label{eq:two-tail-caloric-bound}
		\end{equation}
		Thus Lemma~\ref{lem:homogeneous-annulus-estimate} controls $\nabla h$. To estimate $\nabla g$, we divide $U_a^R$ into dyadic annuli and write $g$ as a sum of solutions with zero parabolic boundary values, each having its right-hand side supported in one of these annuli. We first estimate the supremum norm of each solution and then apply Lemma~\ref{lem:homogeneous-annulus-estimate} where its right-hand side vanishes. Finally, we sum the gradient estimates, using \eqref{eq:two-tail-forcing} to obtain a constant independent of the number of dyadic annuli. All unspecified supremum norms below are taken over $\Omega(a,R)$.
		
		\medskip
		\noindent
		\emph{Step 1. Supremum estimates for a single annulus.} We first consider an equation whose right-hand side is supported in one annulus.
		
		Suppose $G_\rho$ is bounded and supported in $U_\rho^{2\rho}\times[-R^2,0]$, where $a\le\rho<R$, and
		\begin{equation*}
			|G_\rho|\le\frac{b}{\rho^2}.
		\end{equation*}
		Let $w_\rho$ solve
		\begin{equation}
			\begin{cases}
				(\partial_t-\Delta)w_\rho=G_\rho,
				& \text{in } \Omega(a,R),\\
				w_\rho=0,
				& \text{on } \partial\Omega(a,R).
			\end{cases}
			\label{eq:shell-dirichlet-problem}
		\end{equation}
		For bounded right-hand sides, $w_\rho$ is defined using the Dirichlet heat kernel. We derive two bounds for $\|w_\rho\|_\infty$: one involving $\ln(\rho/a)$ and the other involving $\ln(R/\rho)$. These will be used separately for the two terms in \eqref{eq:two-tail-forcing}.
		
		For the first bound, construct a nonnegative radial function $H_\rho$ with $H_\rho(a)=0$ and $-\Delta H_\rho=(b/\rho^2)\mathbf1_{U_\rho^{2\rho}}$. We require $H_\rho'=0$ beyond $2\rho$, so that $H_\rho$ does not continue to grow outside this annulus. Since $\Delta H_\rho=r^{-1}(rH_\rho')'$, integration inward gives
		\begin{equation*}
			rH_\rho'(r)=
			\begin{cases}
				\dfrac{3b}{2},
				& a<r\le\rho,\\[2mm]
				\dfrac{b(4\rho^2-r^2)}{2\rho^2},
				& \rho<r<2\rho,\\[2mm]
				0,
				& r\ge2\rho.
			\end{cases}
		\end{equation*}
		The continuity of $rH_\rho'$ ensures that the equation for $H_\rho$ also holds across $\rho$ and $2\rho$ in the weak sense. Moreover, $H_\rho'\ge0$ and $H_\rho'(r)\le Cb/r$, whence
		\begin{equation*}
			\|H_\rho\|_{L^\infty(U_a^R)}
			\le H_\rho(2\rho)
			\le Cb\left(1+\ln\frac{\rho}{a}\right).
		\end{equation*}
		Comparison of $w_\rho$ and $-w_\rho$ with $H_\rho$ yields
		\begin{equation}
			\|w_\rho\|_\infty
			\le Cb\left(1+\ln\frac{\rho}{a}\right).
			\label{eq:shell-inner-bound}
		\end{equation}
		This estimate is useful when $\rho$ is close to $a$.
		
		For the second bound, extend $|G_\rho|$ by zero to $\mathbb R^2$. For the Euclidean heat kernel $K_\tau$, the bounds $\|G_\rho(\cdot,t)\|_1\le Cb$ and $\|G_\rho(\cdot,t)\|_\infty\le b/\rho^2$ imply
		\begin{equation*}
			\|K_\tau*|G_\rho(\cdot,t)|\|_\infty
			\le C\min\left\{\frac b{\rho^2},\frac b\tau\right\}.
		\end{equation*}
		The Dirichlet heat kernel is bounded above by $K_\tau$. Duhamel's formula and integration over the elapsed time give
		\begin{equation}
			\|w_\rho\|_\infty
			\le Cb\int_0^{R^2}
			\min\{\rho^{-2},\tau^{-1}\}\,d\tau
			\le Cb\left(1+\ln\frac R\rho\right).
			\label{eq:shell-outer-bound}
		\end{equation}
		
		\medskip
		\noindent
		\emph{Step 2. Gradient estimates for a single annulus.} We next estimate $\nabla w_\rho(x,0)$ according to the position of $x$ relative to $U_\rho^{2\rho}$.
		
		Let $B_\rho\ge b$ be an upper bound for $\|w_\rho\|_\infty$, chosen using either \eqref{eq:shell-inner-bound} or \eqref{eq:shell-outer-bound}. Write $r=|x|$. Interior parabolic estimates give
		\begin{equation}
			|\nabla w_\rho(x,0)|
			\le\frac{CB_\rho}{r},
			\qquad \forall x\in U_{4a}^{R/8}.
			\label{eq:shell-local-gradient}
		\end{equation}
		Indeed, choose a cylinder of radius comparable to $r$ around $(x,0)$ and contained in the original domain. If it meets $U_\rho^{2\rho}$, then $r\simeq\rho$, and the right-hand side contributes at most $Cr\|G_\rho\|_\infty\le Cb/r$. If it does not meet that annulus, $w_\rho$ is caloric there.
		
		When $x$ is sufficiently far from $U_\rho^{2\rho}$, we can improve \eqref{eq:shell-local-gradient} by applying Lemma~\ref{lem:homogeneous-annulus-estimate}. If $8\rho\le r\le R/8$, then $w_\rho$ is caloric on $U_{2\rho}^R\times[-R^2,0]$, so
		\begin{equation}
			|\nabla w_\rho(x,0)|
			\le CB_\rho
			\left[
			\frac1{r\ln(r/(2\rho))}+\frac1R
			\right].
			\label{eq:shell-outside-gradient}
		\end{equation}
		If $4a\le r\le\rho/8$, then $w_\rho$ is caloric on $U_a^\rho\times[-\rho^2,0]$, and the same lemma gives
		\begin{equation}
			|\nabla w_\rho(x,0)|
			\le CB_\rho
			\left[
			\frac1{r\ln(r/a)}+\frac1\rho
			\right].
			\label{eq:shell-inside-gradient}
		\end{equation}
		
		\medskip
		\noindent
		\emph{Step 3. Decomposition into dyadic annuli.} Choose $j_{\max}$ so that
		\begin{equation*}
			2^{j_{\max}}a<R\le2^{j_{\max}+1}a,
		\end{equation*}
		and put $\rho_j=2^ja$, $0\le j\le j_{\max}$. To separate the two terms in \eqref{eq:two-tail-forcing}, write
		\[
		q_-=(a/r)^\beta,\qquad q_+=(r/R)^\beta,
		\qquad
		G^\pm=\frac{q_\pm}{q_-+q_+}G.
		\]
		Then $G=G^-+G^+$ and $|G^\pm|\le Lr^{-2}q_\pm$. Let $G_j^\pm$ be the restriction of $G^\pm$ to $U_{\rho_j}^{\min\{2\rho_j,R\}}$. These annuli are disjoint and cover $U_a^R$ up to their boundaries. Thus $G=\sum_j(G_j^-+G_j^+)$ almost everywhere, with
		\begin{equation*}
			|G_j^-|
			\le\frac{C_\beta L}{\rho_j^2}2^{-\beta j},
		\end{equation*}
		\begin{equation*}
			|G_j^+|
			\le\frac{C_\beta L}{\rho_j^2}
			2^{-\beta(j_{\max}-j)}.
		\end{equation*}
		
		Let $w_j^\pm$ solve \eqref{eq:shell-dirichlet-problem} with right-hand side $G_j^\pm$ in place of $G_\rho$. For $w_j^-$, use \eqref{eq:shell-inner-bound} and $\ln(\rho_j/a)=j\ln2$ to obtain
		\begin{equation}
			\|w_j^-\|_\infty
			\le C_\beta L(1+j)2^{-\beta j}
			=:B_j^-.
			\label{eq:two-tail-inner-bound}
		\end{equation}
		For $w_j^+$, use \eqref{eq:shell-outer-bound} and $\ln(R/\rho_j)\le(j_{\max}+1-j)\ln2$ to obtain
		\begin{equation}
			\|w_j^+\|_\infty
			\le C_\beta L(1+j_{\max}-j)2^{-\beta(j_{\max}-j)}
			=:B_j^+.
			\label{eq:two-tail-outer-bound}
		\end{equation}
		Choose the constants so that $B_j^\pm$ also bound $\rho_j^2\|G_j^\pm\|_\infty$; the estimates in Step 2 then apply with $B_\rho=B_j^\pm$. Summing the preceding bounds gives
		\begin{equation}
			\sum_jB_j^-+\sum_jB_j^+\le C_\beta L.
			\label{eq:two-tail-summed-bound}
		\end{equation}
		By linearity and uniqueness for the zero-boundary problem \eqref{eq:two-tail-inhomogeneous-part},
		\begin{equation}
			g=\sum_jw_j^-+\sum_jw_j^+,
			\qquad
			f=h+\sum_jw_j^-+\sum_jw_j^+.
			\label{eq:two-tail-solution-decomposition}
		\end{equation}
		We now estimate the two sums of gradients separately.
		
		\medskip
		\noindent
		\emph{Step 4. Estimating $\sum_j|\nabla w_j^-|$.} Fix $32a\le r\le R/64$ and set $z=r/a$. We divide the sum at $\rho_j\simeq\sqrt{ar}$. For the smaller annuli, we use \eqref{eq:shell-outside-gradient} to obtain the denominator $r\ln(r/a)$. For the remaining annuli, we use \eqref{eq:shell-local-gradient} together with the factor $2^{-\beta j}$ in \eqref{eq:two-tail-inner-bound}.
		
		Specifically, choose
		\begin{equation*}
			j_*:=\left\lfloor\frac12\log_2z\right\rfloor-2
			=\frac12\log_2z+O(1).
		\end{equation*}
		For $j\le j_*$, we have $8\rho_j\le r$ and $\ln(r/(2\rho_j))\ge c\ln z$. Equations \eqref{eq:shell-outside-gradient} and \eqref{eq:two-tail-summed-bound} therefore give
		\begin{equation*}
			\sum_{j\le j_*}|\nabla w_j^-(x,0)|
			\le C_\beta L
			\left[\frac1{r\ln(r/a)}+\frac1R\right].
		\end{equation*}
		For $j>j_*$, summing \eqref{eq:two-tail-inner-bound} yields
		\begin{equation*}
			\sum_{j>j_*}B_j^-
			\le C_\beta L(1+\ln z)z^{-\beta/2}.
		\end{equation*}
		Since
		\begin{equation*}
			(1+\ln z)z^{-\beta/2}
			\le\frac{C_\beta}{\ln z},
			\qquad z\ge32,
		\end{equation*}
		estimate \eqref{eq:shell-local-gradient} gives
		\begin{equation*}
			\sum_{j>j_*}|\nabla w_j^-(x,0)|
			\le\frac{C_\beta L}{r\ln(r/a)}.
		\end{equation*}
		Combining the two sums, we obtain
		\begin{equation}
			\sum_j|\nabla w_j^-(x,0)|
			\le C_\beta L
			\left[\frac1{r\ln(r/a)}+\frac1R\right].
			\label{eq:two-tail-inner-gradient}
		\end{equation}
		
		\medskip
		\noindent
		\emph{Step 5. Estimating $\sum_j|\nabla w_j^+|$.} We divide this sum according to whether $\rho_j\ge8r$. In that case, $w_j^+$ is caloric inside $\rho_j$, so apply \eqref{eq:shell-inside-gradient}:
		\begin{equation}
			\sum_{\rho_j\ge8r}|\nabla w_j^+(x,0)|
			\le
			\frac{C}{r\ln(r/a)}
			\sum_{\rho_j\ge8r}B_j^+
			+C\sum_{\rho_j\ge8r}\frac{B_j^+}{\rho_j}.
			\label{eq:two-tail-outer-large-annuli}
		\end{equation}
		By \eqref{eq:two-tail-summed-bound}, the first term on the right is bounded by $C_\beta L/(r\ln(r/a))$. To estimate the second term, put $\ell=j_{\max}-j$. Then $\rho_j\simeq R2^{-\ell}$, and
		\begin{equation*}
			\frac{B_j^+}{\rho_j}
			\le\frac{C_\beta L}{R}
			(1+\ell)2^{(1-\beta)\ell}.
		\end{equation*}
		Recall that $\gamma=\min\{\beta/2,1/2\}$, so $0<\gamma<\min\{\beta,1\}$. In particular,
		\[
		(1+\ell)2^{(1-\beta)\ell}
		\le C_\beta2^{(1-\gamma)\ell}.
		\]
		Since $\rho_j\ge8r$ implies $\ell\le\log_2(R/r)+C$, summation gives
		\begin{equation}
			\sum_{\rho_j\ge8r}\frac{B_j^+}{\rho_j}
			\le\frac{C_\beta L}{r}
			\left(\frac rR\right)^\gamma.
			\label{eq:two-tail-outer-weighted-sum}
		\end{equation}
		The choice of $\gamma$ allows the same argument for every $\beta>0$, including $\beta=1$.
		
		For $\rho_j<8r$, use \eqref{eq:two-tail-outer-bound} to obtain
		\begin{equation*}
			\sum_{\rho_j<8r}B_j^+
			\le C_\beta L
			\bigl(1+\ln(R/r)\bigr)\left(\frac rR\right)^\beta
			\le C_\beta L\left(\frac rR\right)^\gamma.
		\end{equation*}
		Thus \eqref{eq:shell-local-gradient} gives
		\begin{equation*}
			\sum_{\rho_j<8r}|\nabla w_j^+(x,0)|
			\le\frac{C_\beta L}{r}
			\left(\frac rR\right)^\gamma.
		\end{equation*}
		Together with \eqref{eq:two-tail-outer-large-annuli} and \eqref{eq:two-tail-outer-weighted-sum}, this proves
		\begin{equation}
			\sum_j|\nabla w_j^+(x,0)|
			\le C_\beta L
			\left[
			\frac1{r\ln(r/a)}
			+\frac1r\left(\frac rR\right)^\gamma
			\right].
			\label{eq:two-tail-outer-gradient}
		\end{equation}
		
		\medskip
		Finally, the function $h$ defined in \eqref{eq:two-tail-caloric-part} is caloric and satisfies \eqref{eq:two-tail-caloric-bound}, so Lemma~\ref{lem:homogeneous-annulus-estimate} yields
		\begin{equation}
			|\nabla h(x,0)|
			\le CM
			\left[\frac1{r\ln(r/a)}+\frac1R\right].
			\label{eq:two-tail-caloric-gradient}
		\end{equation}
		As $r\le R$ and $0<\gamma\le1$,
		\begin{equation*}
			\frac1R\le\frac1r\left(\frac rR\right)^\gamma.
		\end{equation*}
		Using \eqref{eq:two-tail-solution-decomposition}, add \eqref{eq:two-tail-inner-gradient}, \eqref{eq:two-tail-outer-gradient}, and \eqref{eq:two-tail-caloric-gradient} to obtain \eqref{eq:two-tail-annulus-gradient}.
	\end{proof}

	\subsection{Neck estimate with moving inner boundary}
	\label{subsec:moving-annulus-estimate}
	
	We now prove Theorem~\ref{thm:moving-annulus-estimate}. We apply Lemma~\ref{lem:two-tail-annulus-estimate} on fixed annuli contained in $\Omega_c(a,1)$, with radii chosen according to the point where the gradient is estimated.
	
	\begin{proof}[Proof of Theorem~\ref{thm:moving-annulus-estimate}]
		We first obtain a pointwise estimate at $(x,0)$ for $x\in U_{K_0a}^{\delta_0}$, choosing $K_0$ sufficiently large and $\delta_0$ sufficiently small below. We then compare the resulting logarithmic term with $\ln(r/a)$ and $\ln(1/r)$, and integrate over $U_{Ka}^{\delta}$.
		
		\medskip
		\noindent
		\emph{Step 1. Applying the fixed-annulus estimate.} Write $r=|x|$ and set
		\begin{equation*}
			R_r:=r^\mu,\qquad
			a_r:=a+cR_r^2=a+cr^{2\mu}.
		\end{equation*}
		The choice of $a_r$ places the fixed inner boundary outside the moving boundary throughout $[-R_r^2,0]$. The conditions $\mu>1/2$ and $\mu<1$ allow $a_r$ to be much smaller than $r$ and $R_r$ to be much larger than $r$. Indeed, for $K_0a<r<\delta_0$,
		\[
		\frac{a_r}{r}
		\le\frac1{K_0}+c\delta_0^{2\mu-1},
		\qquad
		\frac r{R_r}\le\delta_0^{1-\mu}.
		\]
		Choose $K_0$ large and then $0<\delta_0<1$ small, depending only on $c,\mu$, so that
		\begin{equation}\label{eq:moving-annulus-separation}
			a_r\le\frac r{64},
			\quad
			r\le\frac{R_r}{64}.
		\end{equation}
		The fixed cylinder
		\begin{equation*}
			\Omega_r:=\Omega(a_r,R_r)
			=U_{a_r}^{R_r}\times[-R_r^2,0]
		\end{equation*}
		is contained in $\Omega_c(a,1)$, since $a-cs\le a+cR_r^2=a_r$ for $-R_r^2\le s\le0$ and $R_r<1$.
		
		On $\Omega_r$, equation \eqref{eq:moving-annulus-forcing} and $R_r<1$ give
		\begin{align*}
			|(\partial_s-\Delta)f(y,s)|
			&\le\frac L{|y|^2}
			\left[
			\left(\frac{a_r}{|y|}\right)^\beta
			+|y|^\beta
			\right]\notag\\
			&\le\frac L{|y|^2}
			\left[
			\left(\frac{a_r}{|y|}\right)^\beta
			+\left(\frac{|y|}{R_r}\right)^\beta
			\right].
		\end{align*}
		By \eqref{eq:moving-annulus-separation}, Lemma~\ref{lem:two-tail-annulus-estimate} applies at $x$, with inner radius $a_r$ and outer radius $R_r$. Hence
		\begin{equation}
			|\nabla f(x,0)|
			\le C_\beta(M+L)
			\left[
			\frac1{r\ln(r/a_r)}
			+r^{-1+\gamma(1-\mu)}
			\right],
			\label{eq:moving-annulus-fixed-gradient}
		\end{equation}
		where $\gamma=\min\{\beta/2,1/2\}$ as in \eqref{eq:two-tail-exponent}.
		
		\medskip
		\noindent
		\emph{Step 2. Comparing the logarithmic terms.} We show that
		\begin{equation}
			\frac1{\ln(r/a_r)}
			\le C_{c,\mu}
			\left[
			\frac1{\ln(r/a)}
			+\frac1{\ln(1/r)}
			\right].
			\label{eq:moving-annulus-log-comparison}
		\end{equation}
		If $a\ge cr^{2\mu}$, then $a_r\le2a$. For $K_0\ge4$, we have $\ln(r/a_r)\ge\ln(r/a)-\ln2\ge\frac12\ln(r/a)$, so
		\begin{equation*}
			\frac1{\ln(r/a_r)}
			\le\frac2{\ln(r/a)}.
		\end{equation*}
		If $a<cr^{2\mu}$, then $a_r\le2cr^{2\mu}$ and
		\begin{equation*}
			\ln\frac r{a_r}
			\ge(2\mu-1)\ln\frac1r-\ln(2c).
		\end{equation*}
		Since $2\mu-1>0$, decreasing $\delta_0$ if necessary gives
		\begin{equation*}
			\frac1{\ln(r/a_r)}
			\le\frac2{(2\mu-1)\ln(1/r)}.
		\end{equation*}
		Thus \eqref{eq:moving-annulus-log-comparison} follows in both cases. Substituting it into \eqref{eq:moving-annulus-fixed-gradient} yields
		\begin{equation}
			|\nabla f(x,0)|
			\le C(M+L)
			\left[
			\frac1{r\ln(r/a)}
			+\frac1{r\ln(1/r)}
			+r^{-1+\gamma(1-\mu)}
			\right],
			\label{eq:moving-annulus-gradient}
		\end{equation}
		where $C$ depends only on $c,\beta,\mu$.
		
		\medskip
		\noindent
		\emph{Step 3. Integration over the annulus.} Let $K\ge K_0$ and $0<\delta<\delta_0$. If $Ka\ge\delta$, the conclusion is immediate. Otherwise, square \eqref{eq:moving-annulus-gradient} and integrate over $U_{Ka}^{\delta}$ in polar coordinates. The three resulting radial integrals satisfy
		\begin{align*}
			\int_{Ka}^{\delta}
			\frac{dr}{r\ln^2(r/a)}
			&\le\frac1{\ln K},
			\\
			\int_{Ka}^{\delta}
			\frac{dr}{r\ln^2(1/r)}
			&\le\int_0^\delta
			\frac{dr}{r\ln^2(1/r)}
			=\frac1{\ln(1/\delta)},
			\\
			\int_{Ka}^{\delta}
			r^{-1+2\gamma(1-\mu)}\,dr
			&\le
			\frac{\delta^{2\gamma(1-\mu)}}{2\gamma(1-\mu)}.
		\end{align*}
		Using $(A+B+C)^2\le3(A^2+B^2+C^2)$ and including the angular integral proves \eqref{eq:moving-annulus-energy-estimate}.
	\end{proof}
	
	\begin{remark}
		The restrictions $\mu>1/2$ and $\mu<1$ ensure respectively that $a_r-a=cr^{2\mu}=o(r)$ and $r/R_r\to0$. In Section~\ref{sec:neck-estimates}, we take $\mu=3/4$ and $0<\beta<1$. Then $\gamma=\beta/2$ and $2\gamma(1-\mu)=\beta/4$, giving the term $\delta^{\beta/4}$ in Proposition~\ref{prop:normalized-neck-tension}.
	\end{remark}

	\section{Rigidity of antiholomorphic gap theorem and Riemannian analog}
	\label{sec:gap-rigidity}
	
	In this appendix, we prove a geometric characterization of the equality in Theorem~\ref{thm:antiholomorphic-energy-threshold} and Theorem~\ref{thm:antiholomorphic-gap-general} under strictly positive holomorphic bisectional curvature. We also give the corresponding gap and rigidity for Riemannian targets.
	
	Throughout this appendix, $(\Sigma,g,j)$ is a closed connected Riemann surface and $\omega:(\Sigma,g)\to(N,h)$ is smooth and harmonic. We use the notation and curvature convention of Section~\ref{sec:preliminaries}. As in Section~\ref{sec:antiholomorphic-energy-gap}, $\nabla_{\tilde g}$ denotes the connection induced by the Levi-Civita connection of $\tilde g$ and the pull-back connection on $\mathcal E=\omega^*TN$. For a conformal immersion, total geodesicity is equivalent to $\nabla_{\omega^*h}d\omega=0$. We write $S_R^2$ for the round sphere of radius $R$. A conformal branched covering may be holomorphic or antiholomorphic; its degree always denotes the number of sheets. Total geodesicity of a branched immersion is understood away from its branch points.
	
	\subsection{Rigidity of antiholomorphic gap theorem}
	
	As in Section~\ref{sec:antiholomorphic-energy-gap}, let $N$ be compact K\"ahler with nonnegative holomorphic bisectional curvature and $H_*>0$. We keep $\delta_N=4\pi/H_*$ and write $\Phi=\Phi(\omega)$, $\Psi=\Psi(\omega)$, and $Z=\{\Phi=0\}$. For $p\in Z$, let $m_p=\ord_p(\Phi)$ be its order as in Lemma~\ref{lem:harmonic-one-form-zeros}.
	
	First we characterize the equality in Theorem~\ref{thm:antiholomorphic-gap-general} for a general surface. By Remark~\ref{rem:gap-equality}, \eqref{eq:antiholomorphic-zero-order-gap} is an equality precisely when \eqref{eq:gap-rigidity-kahler-equality} holds on $\Sigma\setminus Z$. In particular, $\nabla_{\tilde g}\Phi=0$ makes the image subbundle of $\Phi$ parallel. Strict positivity gives the following geometric conclusion.
	
	\begin{proposition}\label{prop:gap-rigidity-kahler-surface}
		Suppose $N$ has strictly positive holomorphic bisectional curvature. For a nonholomorphic harmonic map $\omega:\Sigma\to N$, equality in \eqref{eq:antiholomorphic-zero-order-gap} holds if and only if $\omega$ is an antiholomorphic totally geodesic branched immersion whose induced metric has constant Gauss curvature $H_*$ away from the branch points.
	\end{proposition}
	\begin{proof}
		Suppose equality holds and set $\tilde g=e_{\bar\partial}g$ on $\Sigma\setminus Z$. In \eqref{eq:gap-rigidity-kahler-equality}, write $X=\Phi(e_1)$ and $Y=\Psi(e_1)$ for a $\tilde g$-orthonormal frame $(e_1,je_1)$. Since $X\ne0$, strict positivity and $B_N(X,Y)e_{\partial,\tilde g}=0$ force $Y=0$. Thus $\Psi=0$ on $\Sigma\setminus Z$ and everywhere by continuity. Hence $\omega$ is antiholomorphic, $d\omega=\Phi$, and $\tilde g=\omega^*h$. The remaining equality conditions give $\nabla_{\tilde g}d\omega=0$ and $H_N(X)=H_*$. Therefore its regular part is totally geodesic and, by the Gauss equation, its induced metric has curvature $H_*$. Lemma~\ref{lem:harmonic-one-form-zeros} identifies the finitely many zeros of $d\omega$ as branch points.
		
		Conversely, for such an antiholomorphic branched immersion, $\tilde g=\omega^*h$ and $K_{\tilde g}=H_*$ off the branch points, which gives equality in \eqref{eq:antiholomorphic-zero-order-gap}.
		
		Lemma~\ref{lem:gap-rigidity-factorization} below also gives $\omega=\iota\circ f$, where $f:\Sigma\to S^2_{H_*^{-1/2}}$ is an antiholomorphic branched covering and $\iota:S^2_{H_*^{-1/2}}\to N$ is a holomorphic totally geodesic isometric immersion, orienting the intermediate sphere by $J$.
	\end{proof}
	
	For $\Sigma=S^2$, we have
	
	\begin{corollary}\label{thm:gap-rigidity-kahler}
		Suppose $N$ has strictly positive holomorphic bisectional curvature. Then a nonholomorphic harmonic sphere $\omega:S^2\to N$ has antiholomorphic energy $E_{\bar\partial}(\omega)=\delta_N$ if and only if $\omega$ is an antiholomorphic totally geodesic immersion with induced metric of curvature $H_*$.
	\end{corollary}
	\begin{proof}
		Since $\chi(S^2)=2$, the identity $E_{\bar\partial}(\omega)=\delta_N$ implies that the equality $K_{\tilde{g}}=H_*$ holds with no branch points. So the corollary follows from Proposition~\ref{prop:gap-rigidity-kahler-surface}.
	\end{proof}
	
	The following factorization applies to the equality cases above and to the Riemannian case below.
	
	\begin{lemma}\label{lem:gap-rigidity-factorization}
		Suppose $\omega:\Sigma\to N$ is nonconstant, conformal and totally geodesic on its regular part, and its image tangent planes have sectional curvature $K>0$. Then $\omega=\iota\circ f$, where $\iota:S^2_{K^{-1/2}}\to N$ is a totally geodesic isometric immersion and $f$ is a finite holomorphic branched covering. Reversing the orientation of the intermediate sphere makes $f$ antiholomorphic.
	\end{lemma}
	\begin{proof}
		Apply Lemma~\ref{lem:harmonic-one-form-zeros} to the harmonic $\mathcal E$-valued 1-form $\xi=d\omega$. Its proof gives, near each zero $p$, $\xi_z=z^{m_p}\eta(z)$ with $\eta$ smooth and $\eta(0)\ne0$. Conformality gives $h_{\mathbb C}(\eta,\eta)=0$, so the real and imaginary parts of $\eta$ are orthogonal and have the same positive length. Hence
		\[
		P_z=\operatorname{span}_{\mathbb R}
		\{\operatorname{Re}\eta,-\operatorname{Im}\eta\}
		\]
		extends the oriented image tangent plane smoothly across the zero. The induced metric has the form $|z|^{2m_p}b(z)|dz|^2$, with $b>0$ smooth.
		
		We next show that the local image at a branch point $p$ is a smooth totally geodesic disk. The length metric of $\omega^*h$ extends across $p$; the displayed local form makes a small closed coordinate disk a compact length space. A point sufficiently close to $p$ can therefore be joined to $p$ by a minimizing curve that stays in its interior. Away from $p$ the curve is an induced-metric geodesic, so its image under $\omega$ is a target geodesic. This geodesic extends to $\omega(p)$, with initial velocity in $P_p$ by continuity of the plane field. The local image thus lies in the smooth exponential disk $C=\exp_{\omega(p)}(P_p\cap B^{T_{\omega(p)}N}_\varepsilon(0))$, where the latter ball is taken in the target tangent space. Choose an orientation-compatible isothermal coordinate $w$ on $C$, with $w(\omega(p))=0$. Then $w\circ \omega$ is holomorphic off $p$ and extends holomorphically across $p$. After a change of source coordinate it has the form $w\circ \omega=z^{m_p+1}$. Its image contains a neighborhood of $\omega(p)$ in $C$, whose second fundamental form consequently vanishes, including at $\omega(p)$ by continuity.
		
		Consider the image $\widehat\Sigma=\{(\omega(p),P_p):p\in\Sigma\}$ in the bundle of oriented tangent two-planes of $N$. Each local image disk lifts to a smooth disk in this bundle. If two such disks meet at the same point and oriented plane, they agree near that point: both are determined by the exponential map on that plane. The local form $z^{m_p+1}$ shows that the fibers of $p\mapsto(\omega(p),P_p)$ are discrete and hence finite. Compactness of $\Sigma$ then ensures that a small neighborhood of any point of $\widehat\Sigma$ contains only the disk germs from its preimages, which all agree. Thus $\widehat\Sigma$ is a smooth compact connected oriented surface, and the natural projection $\iota:\widehat\Sigma\to N$ is a totally geodesic immersion.
		
		Equip $\widehat\Sigma$ with $\iota^*h$. The map $f:p\mapsto(\omega(p),P_p)$ is a finite holomorphic branched covering, with $\omega=\iota\circ f$; its local branching order is $m_p=\ord_p(d\omega)$. By the Gauss equation, $\widehat\Sigma$ has constant curvature $K>0$. A closed oriented surface of constant positive curvature is a round sphere, so $\widehat\Sigma$ is isometric to $S^2_{K^{-1/2}}$, which proves the factorization.
	\end{proof}
	For any such factorization of degree $d$, Riemann--Hurwitz and the area formula give
	\begin{equation}\label{eq:gap-rigidity-degree}
		\chi(\Sigma)+\sum_p\ord_p(d\omega)=2d,
		\qquad E(\omega)=\frac{4\pi d}{K}.
	\end{equation}
	In the antiholomorphic case, $E=E_{\bar\partial}$ and $d\omega=\Phi$. For $\Sigma=S^2$, the total branching order is $2d-2$, so the lowest energy corresponds to $d=1$.

	\subsection{Explicit energy gap for Riemannian targets}
	
	The Riemannian energy gap recorded here is classical and we make no claim of originality for these bounds. The spherical bound $E(\omega)\ge 4\pi/K_*$ follows easily from the theory of harmonic spheres by Sacks-Uhlenbeck~\cite{SacksUhlenbeck1981}, while the general surface estimate, including the zero-order contribution, is contained in the estimates of Moore~\cite{Moore2006}. Here we include a unified conformal-metric proof and an explicit discussion of equality for completeness.
	
	For a Riemannian target, we apply the same argument to $\xi=d\omega$. Write $K_N$ for sectional curvature and $K_*:=\sup_N(K_N)_+$, and suppose $0<K_*<\infty$. No lower curvature bound is needed. We first prove the general closed-surface result and then specialize to $S^2$.
	
	\begin{proposition}\label{prop:gap-rigidity-riemannian-surface}
		Every nonconstant harmonic map $\omega:\Sigma\to N$ satisfies
		\begin{equation}\label{eq:gap-rigidity-riemannian-gap}
			E(\omega)\ge\frac{2\pi}{K_*}
			\left(\chi(\Sigma)+\sum_{p\in Z}m_p\right),
		\end{equation}
		where $Z=\{d\omega=0\}$ and $m_p=\ord_p(d\omega)$. Equality holds if and only if $\omega$ is a conformal totally geodesic branched immersion whose induced metric has constant Gauss curvature $K_*$ away from the branch points.
	\end{proposition}
	\begin{proof}
		The differential $\xi=d\omega$ is a harmonic $\mathcal E$-valued 1-form. On $\Sigma\setminus Z$, set $\tilde g=\tfrac12|\xi|^2g$. Then
		\[
		|\xi|_{\tilde g}^2=2,\qquad
		\operatorname{Area}_{\tilde g}(\Sigma\setminus Z)=E(\omega).
		\]
		Harmonicity is conformally invariant in dimension two, so the Bochner formula with respect to $\tilde g$ gives
		\begin{equation}\label{eq:gap-rigidity-riemannian-bochner}
			0=|\nabla_{\tilde g}\xi|_{\tilde g}^2+2K_{\tilde g}
			-2h(R^N(X,Y)Y,X),
		\end{equation}
		where $X=\xi(e_1)$ and $Y=\xi(e_2)$ for a $\tilde g$-orthonormal frame $(e_1,e_2)$. Since $|X|^2+|Y|^2=2$, we obtain
		\begin{equation}\label{eq:gap-rigidity-riemannian-curvature}
			\begin{aligned}
				K_{\tilde g}
				&\le h(R^N(X,Y)Y,X)\\
				&\le K_*\bigl(|X|^2|Y|^2-h(X,Y)^2\bigr)
				\le K_*.
			\end{aligned}
		\end{equation}
		Integrating and applying Lemma~\ref{lem:harmonic-one-form-zeros} to $\xi$ yields
		\[
		2\pi\left(\chi(\Sigma)+\sum_{p\in Z}m_p\right)
		=\int_{\Sigma\setminus Z}K_{\tilde g}\,dV_{\tilde g}
		\le K_*E(\omega).
		\]
		
		Equality holds precisely when $K_{\tilde g}=K_*$ on $\Sigma\setminus Z$. Equality throughout \eqref{eq:gap-rigidity-riemannian-curvature}, together with \eqref{eq:gap-rigidity-riemannian-bochner}, is equivalent to
		\[
		\begin{gathered}
			|X|=|Y|=1,\qquad h(X,Y)=0,\\
			K_N(\operatorname{span}\{X,Y\})=K_*,\qquad
			\nabla_{\tilde g}\xi=0.
		\end{gathered}
		\]
		Thus $\tilde g=\omega^*h$ and the regular part of $\omega$ is conformal and totally geodesic with tangent-plane curvature $K_*$. By the Gauss equation, this equals the curvature of the induced metric. Conversely, for a conformal totally geodesic branched immersion with induced curvature $K_*$, the same equation gives tangent-plane curvature $K_*$, so every inequality is an equality. Lemma~\ref{lem:gap-rigidity-factorization} gives the global factorization $\omega=\iota\circ f$, where $f:\Sigma\to S^2_{K_*^{-1/2}}$ is a conformal branched covering and $\iota:S^2_{K_*^{-1/2}}\to N$ is a totally geodesic isometric immersion. The degree formulas are \eqref{eq:gap-rigidity-degree}.
	\end{proof}
	
	As in the K\"ahler case, the spherical threshold corresponds to degree one.
	
	\begin{corollary}\label{thm:gap-rigidity-riemannian}
		Every nonconstant harmonic sphere $\omega:S^2\to N$ satisfies $E(\omega)\ge4\pi/K_*$, with equality if and only if $\omega$ is a conformal totally geodesic immersion with induced metric of constant Gauss curvature $K_*$.
	\end{corollary}
	\begin{proof}
		Set $\chi(S^2)=2$ in \eqref{eq:gap-rigidity-riemannian-gap}. At the threshold all zero orders vanish, so the proposition gives the claimed immersion. Conversely, Gauss--Bonnet for the induced metric gives $E(\omega)=4\pi/K_*$. The factorization in the proof of the proposition has degree one by \eqref{eq:gap-rigidity-degree}.
	\end{proof}

\end{document}